\documentclass[11pt]{amsart}

\newtheorem{theorem}{Theorem}[section]

\newtheorem{corollary}[theorem]{Corollary}

\newtheorem{lemma}[theorem]{Lemma}

\newtheorem{problem}[theorem]{Problem}
\newtheorem{proposition}[theorem]{Proposition}

\usepackage[colorlinks, bookmarks=true]{hyperref}
\usepackage{color,graphicx,shortvrb}
\usepackage[active]{srcltx} 

\def\11{\textbf{$1$}}
\def\CC{{\mathbb{C}}}

\usepackage{enumerate}
\usepackage{amsmath}
\usepackage{amssymb}
\usepackage[latin 1]{inputenc}

\title[Preservers of $\lambda$-Aluthge transforms]{Preservers of $\lambda$-Aluthge transforms}

\author[A. Ben Ali Essaleh]{Ahlem Ben Ali Essaleh}
\email{ahlem.benalisaleh@gmail.com}
\address{D{\'e}partement de Math{\'e}matiques, Institut Pr{\'e}paratoire aux Etudes d'Ing{\'e}nieurs de Gafsa, Universit{\'e} de Gafsa, 2112 Gafsa, Tunisia.}

\author[A.M. Peralta]{Antonio M. Peralta}
\email{aperalta@ugr.es}
\address{Departamento de An{\'a}lisis Matem{\'a}tico, Facultad de
Ciencias, Universidad de Granada, 18071 Granada, Spain.}



\begin{document}

\maketitle

\begin{abstract} Let $M$ and $N$ be arbitrary von Neumann algebras. For any $a$ in $M$ or in $N$, let $\Delta_{\lambda}(a)$ denote the $\lambda$-Aluthge transform of $a$. Suppose that $M$ has no abelian direct summand. We prove that every bijective map $\Phi:M\to N$ satisfying $$\Phi(\Delta_{\lambda}(a\circ b^*))=\Delta_{\lambda}(\Phi(a) \circ \Phi(b)^*), \hbox{ for all } a,\;b\in M,$$ (for a fixed $\lambda\in [0,1]$), maps the hermitian part of $M$ onto the hermitian part of $N$ (i.e. $\Phi (M_{sa}) = N_{sa}$) and its restriction $\Phi|_{M_{sa}} : M_{sa}\to N_{sa}$ is a Jordan isomorphism. If we also assume that $\Phi (x +i y ) = \Phi (x) +\Phi (i y)$ for all $x,y\in M_{sa}$, then there exists a central projection $p_c$ in $M$ such that $\Phi|_{p_c M}$ is a complex linear Jordan $^*$-isomorphism and $\Phi|_{(\textbf{1}-p_c) M}$ is a conjugate linear Jordan $^*$-isomorphism.

Given two complex Hilbert spaces $H$ and $K$ with dim$(H)\geq 2$, we also establish that every bijection $\Phi: \mathcal{B}(H)\to \mathcal{B}(K)$ satisfying $$\Phi(\Delta_{\lambda}(a b^*))=\Delta_{\lambda}(\Phi(a) \Phi(b)^*), \hbox{ for all } a,\;b\in \mathcal{B}(H),$$ must be a complex linear or a conjugate linear $^*$-isomorphism.
\end{abstract}

\section{Introduction}

Let $a$ be an element in a von Neumann algebra $M$. Let $a = u |a|$ be the polar decomposition of $a$ in $M$, where $u$ is a partial isometry in $M$, $|a|= (a^* a)^{\frac12}$, and $u^* u$ is the range projection of $|a|$ (see \cite[Theorem 1.12.1]{S}). Given $\lambda \in[0,1],$ the \emph{$\lambda$-Aluthge transform} of $a$ is defined by
$$\Delta_{\lambda}(a)=|a|^{\lambda} u |a|^{1-\lambda}.$$ It should be be noted that the polar decomposition does not change when computed with respect to any von Neumann subalgebra $N$ of $M$ containing the element $a,$ actually $u$ is precisely the limit in the weak$^*$-topology of $M$ of the sequence $\left(a (\frac1n + |a|)^{-1}\right)_n,$ and hence $u \in N$ (compare \cite[Theorem 1.12.1]{S} or \cite{Ped}). For $\lambda=0$ we have $\Delta_{0}(a) =a$, while for $\lambda = \frac12$ we rediscover the original transform introduced by A. Aluthge in \cite{Aluthge90}. So, we can understand $\Delta_0$ as the identity transform on every C$^*$-algebra. \smallskip

It should be remarked here that any complex linear or conjugate linear $^*$-isomorphism $T$ between von Neumann algebras $M$ and $N$ preserves $\lambda$-Aluthge transforms, that is, \begin{equation}\label{eq linear or conjugate linear *isomorphisms preserve lambda Aluthge transforms} T(\Delta_{\lambda}(a)) = \Delta_{\lambda}(T(a)) \hbox{ for every } a\in M.
\end{equation}

Recently, in \cite{BoteMolNag2016} F. Botelho, L. Moln{\'a}r and G. Nagy studied those linear bijections between von Neumann factors preserving $\lambda$-Aluthge transforms. The concrete result reads as follows: Let $\mathcal{A}$ and $\mathcal{B}$ be von Neumann factors, $\lambda$ a fixed real number in $(0,1]$, and $\Phi : \mathcal{A}\to \mathcal{B}$ a linear bijection. Then the following assertions hold:
\begin{enumerate}[$(a)$]\item If $\mathcal{A}$ is not of type I$_2$, then $\Phi$ preserves the $\lambda$-Aluthge transform (i.e. $\Phi ( \Delta_{\lambda}(a)) = \Delta_{\lambda}(\Phi(a))$ for all $a\in \mathcal{A}$) if and only if there is a $^*$-isomorphism $\Theta : \mathcal{A}\to \mathcal{B}$ and a nonzero scalar $c\in \mathbb{C}$ such that $\Phi (a) = c \Theta (a)$, for all $a\in \mathcal{A}$;
\item If $\mathcal{A}$ is of type I$_2$, then $\mathcal{B}$ is also of type I$_2$ and, without loss of generality, we can assume
that they both coincide with the algebra of all 2 by 2 complex matrices. Then $\Phi$ preserves the $\lambda$-Aluthge transform if
and only if it is either of the form $$ \Phi(a) = c \ v a v^*, \hbox{ for all } a\in \mathcal{A},$$ or of the form
$$\Phi ( a ) = c \left(v a^t v^* - (\hbox{Tr}(a)) \textbf{1}\right), \hbox{ for all } a\in \mathcal{A},$$
where $v$ is a unitary, $c$ is a nonzero scalar, $^t$ stands for the transpose and Tr stands for the
usual trace functional on matrices.
\end{enumerate}

Let $H$ and $K$ be complex Hilbert spaces. When preservers of the $\lambda$-Aluthge transforms are particularized to bijections between von Neumann factors of the form $\mathcal{B}(H)$ and $\mathcal{B}(K)$, F. Chabbabi relaxes the hypothesis concerning the linearity of the mapping at the cost of assuming certain commutativity of the mapping and the $\lambda$-Aluthge transform on products. More concretely, for a given $\lambda\in (0,1)$ and dim$(H)\geq 2$, a bijective mapping $\Phi : \mathcal{B}(H)\to \mathcal{B}(K)$ satisfies \begin{equation}\label{eq product commuting maps} \Phi(\Delta_{\lambda}(a b))=\Delta_{\lambda}(\Phi(a) \Phi(b)),\ \  \text{ for all } a,\;b\in \mathcal{B}(H),
 \end{equation} if and only if $\Phi$ is a complex linear or conjugate linear $^*$-isomorphism (see \cite{Chabbabi2017product}).\smallskip

If in \eqref{eq product commuting maps} the usual product is replaced with the natural Jordan product, $a \circ b := \frac12 ( a b + ba)$, F. Chabbabi and M. Mbekhta establish in \cite[Theorem 1.1]{ChabbabiMbekhta2017Jordan} that, for a given $\lambda\in (0,1)$ and dim$(H)\geq 2$, a bijective mapping $\Phi : \mathcal{B}(H)\to \mathcal{B}(K)$ satisfies \begin{equation}\label{eq Jordan product commuting maps} \Phi(\Delta_{\lambda}(a\circ b))=\Delta_{\lambda}(\Phi(a)\circ \Phi(b)),\ \  \text{ for all } a,\;b\in \mathcal{B}(H),
 \end{equation} (respectively, \begin{equation}\label{eq Jordan *product commuting maps} \Phi(\Delta_{\lambda}(a\circ b^*))=\Delta_{\lambda}(\Phi(a)\circ \Phi(b)^*),\ \  \text{ for all } a,\;b\in \mathcal{B}(H),)
 \end{equation} if and only if $\Phi$ is a complex linear or conjugate linear $^*$-isomorphism.\smallskip

For some reason which remains unknown to these authors, F. Chabbabi and M. Mbekhta do not consider in \cite{Chabbabi2017product, ChabbabiMbekhta2017Mediterr, ChabbabiMbekhta2017Jordan} bijections commuting with the $\lambda$-Aluthge on more natural C$^*$-products in the following sense: Let $\Phi: M\to N$ be a bijection between von Neumann algebras satisfying the following property
\begin{equation}\label{eq1.1}
\Phi(\Delta_{\lambda}(a b^*))=\Delta_{\lambda}(\Phi(a) \Phi(b)^*),\quad \text{ for all }a,\;b\in M,
\end{equation} for some fixed $\lambda \in [0,1].$\smallskip

The study on maps preserving certain $\lambda$-Aluthge transforms can be connected with another interesting line about preservers. Let $M$ and $N$ be von Neumann algebras. In the case $\lambda =0$ the bijections $\Phi : M\to N$ satisfying \eqref{eq product commuting maps} or \eqref{eq Jordan product commuting maps} and the property $\Phi (a^* ) =\Phi (a)^*$, for all $a$ in $M,$ have been studied and described by J. Hakeda \cite{Hak86a, Hak86b} and J. Hakeda and K. Sait\^{o} \cite{HakSaito86}. Let us revisit their main conclusions.\smallskip

According to the usual notation (see \cite{Hopenwasser73}), a \emph{system of $n \times n$ matrix units} in a unital C$^*$-algebra $A$ with unit $\textbf{1}$ is a family of elements $\{u_{ij}: i,j=1,\ldots,n \}\subset A$ satisfying the following properties:\begin{enumerate}[$(a)$]\item $u_{ij}^* = u_{ji}$ for all $i,j$;
\item $u_{ij} u_{kl} = \delta{jk} u_{il}$ for all $i,j,k,l$;
\item $\displaystyle \sum_{i=1}^n u_{ii} = \textbf{1}$.
\end{enumerate}

\begin{theorem}\label{thm Hakeda 86}\cite[Theorem 3.6 and Corollary 3.7]{Hak86a} Let $M$ and $N$ be C$^*$-algebras and $M$ has an $n \times n$ $(n \geq 2)$ matrix unit.
Let $\Phi : M\to N$ be a bijection satisfying $\Phi (x y ) = \Phi (x) \Phi (y)$ {\rm(}equivalently, $\Phi(\Delta_{0}(x y))=\Delta_{0}(\Phi(x) \Phi(y))$, that is, $\Phi$ satisfies \eqref{eq product commuting maps} with $\lambda=0${\rm)} and $\Phi (x^*) = \Phi (x)^*$, for all $x,y\in M$.  Then $\Phi$ is an orthogonal sum of a linear and a conjugate linear $^*$-isomorphism. The same conclusion hold when $M$ and $N$ are von Neumann algebras (or more generally, AW$^*$-algebras) and $M$ has
no abelian summand.$\hfill\Box$
\end{theorem}

\begin{theorem}\label{thm Hakeda Saito 86}\cite[THEOREM and COROLLARY]{HakSaito86} Let $M$ be a C$^*$-algebra, $N$ an associative $^*$-algebra, and suppose that $M$ has a system of $n \times n$  matrix units with $n \geq 2$.
Let $\Phi : M\to N$ be a bijection satisfying $\Phi (x \circ y ) = \Phi (x) \circ \Phi (y)$ {\rm(}equivalently, $\Phi(\Delta_{0}(x\circ y))=\Delta_{0}(\Phi(x) \circ \Phi(y))$, that is, $\Phi$ satisfies \eqref{eq Jordan product commuting maps} with $\lambda=0${\rm)} and $\Phi (x^*) = \Phi (x)^*$, for all $x,y\in M$.  Then $\Phi$ is a real linear Jordan $^*$-isomorphism. The same conclusion holds when $M$ is a von Neumann algebra {\rm(}or more generally an AW$^{*}$-algebra{\rm)} which has no abelian direct summand, and $N$ is a C$^{*}$-algebra.$\hfill\Box$
\end{theorem}

Let us fix a C$^*$-algebra $M$ admitting a system of $n \times n$  matrix units with $n \geq 2$, and another (unital) C$^*$-algebra $N$. If a bijection $\Phi : M\to N$ satisfies \eqref{eq1.1} (respectively, \eqref{eq Jordan *product commuting maps}) with $\lambda =0$, then $$\Phi (x y^* ) = \Phi (  \Delta_0(x y^* ) ) = \Delta_0(\Phi(x) \Phi(y)^*) = \Phi(x) \Phi(y)^*,$$ (respectively, $$\Phi (x \circ y^* ) = \Phi (  \Delta_0(x \circ y^* ) ) = \Delta_0(\Phi(x)\circ \Phi(y)^*) = \Phi(x) \circ \Phi(y)^*{\rm)},$$ for all $x,y\in M$. It can be easily deduced that, in this case we have $\Phi (x) = \Phi (x \textbf{1}^*) = \Phi (x) \Phi (\textbf{1})^*$ (respectively, $\Phi (x) = \Phi (x\circ \textbf{1}^*) = \Phi (x) \circ \Phi (\textbf{1})^*$), for all $x\in M$. This implies, in any case, that $\Phi (\textbf{1}) =\textbf{1}$, and a new application of \eqref{eq1.1} (respectively, \eqref{eq Jordan *product commuting maps}) gives $\Phi (x^*) = \Phi (\textbf{1} x^*) = \Phi (\textbf{1}) \Phi (x)^* = \Phi (x)^*$ (respectively, $\Phi (x^*) = \Phi (\textbf{1}\circ x^*) = \Phi (\textbf{1}) \circ \Phi (x)^* = \Phi(x)^*$), for all $x$ in $M$, and consequently, $\Phi (x y) = \Phi (x (y^*)^*) = \Phi (x) \Phi (y^*)^* = \Phi (x) \Phi (y)$ (respectively, $\Phi (x\circ y) = \Phi (x\circ (y^*)^*) = \Phi (x) \circ \Phi (y^*)^* = \Phi (x) \circ \Phi (y)$). Therefore, the next corollary is a straight consequence of the above Theorems \ref{thm Hakeda 86} and \ref{thm Hakeda Saito 86}.

\begin{corollary}\label{c Hakeda Saito 86 for (4) or (3)} Let $M$ and $N$ be C$^*$-algebras, and suppose that $M$ admits a system of $n \times n$  matrix units with $n \geq 2$.
Let $\Phi : M\to N$ be a bijection satisfying \eqref{eq1.1} {\rm(}respectively, \eqref{eq Jordan *product commuting maps}{\rm)} with $\lambda=0$. Then $\Phi$ is a real linear $^*$-isomorphism {\rm(}respectively, a real linear Jordan $^*$-isomorphism{\rm)}. The same conclusion holds when $M$ is a von Neumann algebra {\rm(}or more generally an AW$^{*}$-algebra{\rm)} which has no abelian direct summand, and $N$ is a C$^{*}$-algebra.
\end{corollary}

In this paper we are interested in the following problems.
\begin{problem}\label{problems} Let $M$ and $N$ be von Neumann algebras such that $M$ has no abelian direct summand. Suppose $\Phi : M\to N$ is a bijection satisfying one of the next properties:
\begin{enumerate}[$(h.1)$]\item $\Phi$ satisfies \eqref{eq product commuting maps} for a fixed $\lambda\in [0,1]$, that is, $$\Phi(\Delta_{\lambda}(a b))=\Delta_{\lambda}(\Phi(a) \Phi(b)), \hbox{ for all } a,\;b\in M,$$  and $\Phi (a^*) = \Phi(a)^*$ for all $a\in M$;
\item $\Phi$ satisfies \eqref{eq Jordan product commuting maps} for a fixed $\lambda\in [0,1],$ that is, $$\Phi(\Delta_{\lambda}(a \circ b))=\Delta_{\lambda}(\Phi(a)\circ \Phi(b)), \hbox{ for all } a,\;b\in M,$$ and $\Phi (a^*) = \Phi(a)^*$ for all $a\in M$;
\item $\Phi$ satisfies \eqref{eq1.1} for a fixed $\lambda\in [0,1]$, that is, $$\Phi(\Delta_{\lambda}(a b^*))=\Delta_{\lambda}(\Phi(a) \Phi(b)^*), \hbox{ for all } a,\;b\in M;$$
\item $\Phi$ satisfies \eqref{eq Jordan *product commuting maps} for a fixed $\lambda\in [0,1]$, that is, $$\Phi(\Delta_{\lambda}(a \circ b^*))=\Delta_{\lambda}(\Phi(a)\circ \Phi(b)^*), \hbox{ for all } a,\;b\in M.$$
\end{enumerate} Is $\Phi$ additive? We can, further, ask whether $\Phi$ is linear.
\end{problem}

It should be noted that if a bijection  $\Phi : M\to N$ satisfies the hypothesis $(h.k)$ ($k\in \{1,2,3,4\})$, then $\Phi^{-1}$ also satisfies the same hypothesis.\smallskip

In the hypothesis of Problem \ref{problems} above. Suppose $\Phi : M\to N$ is a bijection satisfying $(h.1)$ (respectively, $(h.2)$). In this case we have $$\Delta_{\lambda}(\Phi(a) \Phi(b)^*)= \Delta_{\lambda}(\Phi(a) \Phi(b^*)) = \Phi(\Delta_{\lambda}(a b^*))$$ (respectively, $$\Delta_{\lambda}(\Phi(a) \circ \Phi(b)^*)= \Delta_{\lambda}(\Phi(a) \circ \Phi(b^*)) = \Phi(\Delta_{\lambda}(a\circ b^*))),$$ which shows that $\Phi$ satisfies $(h.3)$ (respectively, $(h.4)$). We can therefore reduce to the cases in which hypothesis $(h.3)$ or $(h.4)$ holds.\smallskip

If $\Phi : M\to N$ is a bijection satisfying $(h.3)$ (respectively, $(h.4)$) with $\lambda =0$, the arguments before Corollary \ref{c Hakeda Saito 86 for (4) or (3)} show that $\Phi$ is unital and satisfies $(h.1)$ (respectively, $(h.2)$). So, the problems with hypothesis $(h.1)$ and $(h.3)$ (respectively, $(h.2)$ and $(h.4)$) are equivalent for $\lambda =0$. They have all been solved in this case (compare Corollary \ref{c Hakeda Saito 86 for (4) or (3)}).\smallskip

We should remark that a complex linear or conjugate linear $^*$-isomorphism $T: M\to N$ clearly satisfies the following identities:
\begin{equation}\label{eq linear or conjugate linear *isomorphisms preserves lambda aluthge products} \Delta_{\lambda}(\Phi(a) \Phi(b))\!=\! \Delta_{\lambda}(\Phi(a b))\!=\! \Phi(\Delta_{\lambda}(a b)), \Delta_{\lambda}(\Phi(a) \Phi(b)^*)\!=\! \Phi(\Delta_{\lambda}(a b^*)),
\end{equation} $$\Delta_{\lambda}(\Phi(a)\circ \Phi(b))=  \Phi(\Delta_{\lambda}(a\circ b)), \Delta_{\lambda}(\Phi(a)\circ \Phi(b)^*)= \Phi(\Delta_{\lambda}(a\circ b^*)),$$ for all $a,b\in M$ (compare \eqref{eq linear or conjugate linear *isomorphisms preserve lambda Aluthge transforms}).\smallskip

Our main achievements can be organized in essentially two different blocks. In section \ref{sec:3} we study bijections between general von Neumann algebras $M$ and $N$. We prove that if $M$ has no abelian direct summand, and  $\Phi:M\to N$ is a bijective map satisfying hypothesis $(h.4)$ in Problem \ref{problems}, that is, $$\Phi(\Delta_{\lambda}(a\circ b^*))=\Delta_{\lambda}(\Phi(a)\circ \Phi(b)^*), \hbox{ for all } a,\;b\in M,$$ (for a fixed $\lambda\in [0,1]$), then $\Phi (M_{sa}) = N_{sa}$ and the restriction $\Phi|_{M_{sa}} : M_{sa}\to N_{sa}$ is a Jordan isomorphism (see Theorem \ref{t linearity on hermitian for h4}).\smallskip

In order to understand the behavior of the mapping $\Phi$ on the whole von Neumann algebra $M$, we appeal to an extra assumption, which was already considered by J.F. Aarnes \cite{Aarnes70} and L.J. Bunce, J.D.M. Wright \cite{BunWri96} in their studies on quasi-states and quasi-linear maps. More concretely, if we also assume that $\Phi (x +i y ) = \Phi (x) +\Phi (i y)$ for all $x,y\in M_{sa}$, then there exists a central projection $p_c$ in $M$ such that $\Phi|_{p_c M}$ is a complex linear Jordan $^*$-isomorphism and $\Phi|_{(\textbf{1}-p_c) M}$ is a conjugate linear Jordan $^*$-isomorphism (see Theorem \ref{t linearity on hermitian for h4}).\smallskip

The maps $\Phi:\mathbb{C}\to \mathbb{C}$, $\Phi(z) = z^{-1}$ if $z\neq 0,$ and $\Phi (0) =0,$ and  $\Psi:\mathbb{C}\to \mathbb{C}$, $\Psi(z) = z |z|$ are both bijective, both commute with the natural involution and both preserve products. These examples show that the restriction concerning $M$ in the hypothesis of Theorem \ref{t linearity on hermitian for h4} can not be relaxed.\smallskip

In section \ref{sec:4} we deal with the study of those bijections $\Phi: \mathcal{B}(H)\to \mathcal{B}(K)$ satisfying $(h.3)$ in Problem \ref{problems}, that is, bijections for which there exists $\lambda\in [0,1]$ such that $$\Phi(\Delta_{\lambda}(a b^*))=\Delta_{\lambda}(\Phi(a) \Phi(b)^*), \hbox{ for all } a,\;b\in \mathcal{B}(H),$$ where $H$ and $K$ are complex Hilbert spaces. The counterexamples given in the previous paragraph justifies that we must assume that dim$(H)\geq 2$. Under this assumption we show that every bijection $\Phi:\mathcal{B}(H)\to \mathcal{B}(K)$ satisfying hypothesis $(h.3)$ in Problem \ref{problems} (for a fixed  $\lambda$ in $[0,1]$) must be a complex linear or a conjugate linear $^*$-isomorphism (see Theorem \ref{thm h3 in problems}).

\section{Generalities on $\lambda$-Aluthge transforms}

This section is devoted to gather some of the basic facts and properties of the $\lambda$-Aluthge transform.\smallskip

An element $a$ in a von Neumann algebra $M$ is called \emph{quasi-normal} if $a(a^* a) = (a^* a) a$. It is known that $a$ is quasi-normal if and only if $|a|$ commutes with the partial isometry appearing in the polar decomposition of $a$ (compare \cite[Lemma 4.1]{Brow52}). Furthermore, by \cite[Proposition 1.10]{JungKoPearcy2000} an element $a\in M$ is quasi-normal if and only if $\Delta_{\frac12} (a) = a$. Let us observe that the result in \cite{Brow52} (respectively, \cite{JungKoPearcy2000}) is established in the case $M= \mathcal{B}(H)$ (respectively, $M=\mathcal{B}(H)$ and $\lambda=\frac12$), however, every von Neumann algebra $M$ can be viewed as a weak$^*$-closed C$^*$-subalgebra of some $\mathcal{B}(H)$ (see \cite[\S 3.9]{Ped}), and an element $a\in M$ is quasi-normal if and only if it is quasi-normal in $\mathcal{B}(H)$. Actually, the arguments in the proof of \cite[Proposition 1.10]{JungKoPearcy2000} are valid to prove that for each $\lambda\in (0,1)$, \begin{equation}\label{eq quasi-normal equiv to fixed point of the lambda-Aluthge transform} \hbox{$a$ in $M$ is quasi-normal if and only if $\Delta_{\lambda} (a) = a$.}
\end{equation} The equivalence in \eqref{eq quasi-normal equiv to fixed point of the lambda-Aluthge transform} trivially holds for $\lambda=1$. Namely, let $a=u |a|$ be the polar decomposition of $a$, then $\Delta_{1} (a) = a$ if and only if $|a| u = u |a|$ if and only if $|a|$ and $u$ commute. \smallskip


In \cite[Lemma 2.3]{Chabbabi2017product} F. Chabbabi proves that for every operator $a\in \mathcal{B}(H)$, every projection $p\in \mathcal{B}(H),$ and $0<\lambda< 1$, we have $$ \Delta_{\lambda} (a p ) = a \hbox{ if and only if } a = p a = a p \hbox{ and } a \hbox{ is quasi-normal}.
$$ Since every von Neumann algebra $M$ can be regarded as a weak$*$-closed subalgebra of some $\mathcal{B}(H)$, polar decompositions and $\lambda$-Aluthge transforms do not change when they are computed in $M$ or in $\mathcal{B}(H),$ and quasi-normal elements in $M$ are precisely the quasi-normal elements in $\mathcal{B}(H)$ belonging to $M$, the statement of the next proposition for $0<\lambda<1$ follows from \cite[Lemma 2.3]{Chabbabi2017product}.

\begin{lemma}\label{l Chabb 2.3 von Neumann} Let $a$ be an arbitrary element in a von Neumann algebra $M$, let $p$ be a projection in $M$ and let $\lambda$ be an element in $(0,1]$. Then $ \Delta_{\lambda} (a p ) = a$ if and only if $a = p a = a p$ and $ a$  is quasi-normal.
\end{lemma}

\begin{proof} The comments preceding this lemma assure that for $0<\lambda<1$ the statement follows from \cite[Lemma 2.3]{Chabbabi2017product}. Let us assume that $\Delta_{1} (a p ) = a$.  We may assume that $\|a\|=1$. Let $v |ap| = ap$, be the polar decomposition of $ap$. It follows from the hypothesis that $ |ap| v = a$. Let us observe that $|ap|^2 = p a^* a p \leq \|a\|^2 p = p$. Since, by definition, $v^* v$ is the range projection of $|ap|$ in $M$ (which also coincides with the range projection of $|ap|^2$), we deduce that $v^* v\leq p$. On the other hand, by the hypothesis, we get $$a^* a = v^* |ap|^2 v \leq v^* v\leq p,$$ and
$$ a a^* = |ap| v v^* |ap| \leq |ap|^2 \leq p.$$ This is enough to conclude that $a p = pa =a$, and consequently, $a=\Delta_1(ap) = \Delta_1(a),$ and thus $a$ is quasi-normal (compare \eqref{eq quasi-normal equiv to fixed point of the lambda-Aluthge transform}).
\end{proof}

Given an operator $a\in \mathcal{B}(H)$ and $0<\lambda\leq 1$, it is shown in \cite[Lemma 2]{BoteMolNag2016} that $\Delta_{\lambda} (a) =0$ if and only if $a^2$. Since every von Neumann algebra $M$ can be regarded as a weak$*$-closed subalgebra of some $\mathcal{B}(H)$, we can easily deduce that for each $a$ in $M$ we have \begin{equation}\label{eq kernel Delta lambda are nilpotent} \Delta_{\lambda} (a) =0\Longleftrightarrow a^2=0.
\end{equation}

Throughout the paper, for each C$^*$-algebra $A$, the symbols Proj$(A),$ $A_{sa}$ and $A^+$ will stand for the lattice of all projections in $A,$ the real subspace of all hermitian elements in $A$, and the cone of positive elements in $A$, respectively. The lattice Proj$(A)$ is equipped with the natural partial order given by $p \leq q$ if $p q =p$. Elements $p,q$ in  $\hbox{Proj}(A)$ are called \emph{orthogonal} (written $p\perp q$) if $pq =0$. it can be easily seen that $p\leq q$ if and only if $q-p$ is a projection with $q-p\perp p$.\smallskip

A nonzero projection $p$ in $A$ is called \emph{minimal} if $p A p = \mathbb{C} p$. When $A$ is a von Neumann algebra, a nono-zero projection $p$ in $A$ is minimal if and only if $p$ is minimal with respect to the partial order in Proj$(A)$, that is, $0\neq q\leq p$ implies $p =q$.\smallskip

We begin our study by gathering the basic properties of the maps under study.

\begin{proposition}\label{p basic properties} Let $\Phi:M\to N$ be a bijective map between von Neumann algebras satisfying hypothesis $(h.3)$ {\rm(}respectively, $(h.4)${\rm)} in Problem \ref{problems}. Then the following statements hold:
\begin{enumerate}[$(a)$] \item $\Phi(0)=0;$
\item For each $a\in M,$ we have $\Phi(a a^*)=\Phi(a)\Phi(a)^*$ {\rm(}respectively, $\Phi(a\circ a^*)=\Phi(a)\circ \Phi(a)^*$ {\rm)}.
In particular, $\Phi (M^+) = N^+$;
\item $\Phi$ preserves projections and $\Phi (\hbox{Proj}(M)) = \hbox{Proj} (N)$;
\item $\Phi|_{\hbox{Proj}(M)}: \hbox{Proj}(M)\to \hbox{Proj}(N)$ is an order isomorphism;
\item $\Phi(\textbf{1})=\textbf{1};$
\item $\Phi|_{\hbox{Proj}(M)}: \hbox{Proj}(M)\to \hbox{Proj}(N)$ preserves orthogonality in both directions, that is, $$p\perp q \hbox{ in } M \Longleftrightarrow \Phi(p)\perp \Phi(q) \hbox{ in } N;$$
\item $\Phi$ preserves minimal projections in both directions;
\item If $p_1,\ldots,p_m$ are mutually orthogonal projections in $M$, then the identity $$\displaystyle \Phi\left(\sum_{j=1}^{m} p_j\right)=\sum_{j=1}^{m} \Phi(p_j)$$ holds;
\item If $p$ and $q$ are projections in $M$ with $p\leq q$ then $\Phi (q-p) = \Phi (q) -\Phi(p)$. In particular, $\Phi (\textbf{1}-p) = \textbf{1}-\Phi(p)$ for every projection $p\in M$;
\item $\Phi$ preserves $\lambda$-Aluthge transforms, that is, $\Phi (\Delta_{\lambda} (a)) =\Delta_{\lambda} (\Phi(a)),$ for all $a$ in $M$. In particular, $\Phi$ preserves the set of quasi-normal elements in both directions;
\item $\Phi(\Delta_{\lambda}(a^*))=\Delta_{\lambda}(\Phi(a)^*),$ for all $a\in M.$
\end{enumerate}
\end{proposition}

\begin{proof} Before dealing with the concrete arguments, we observe that for $\lambda=0$ all the statements are straight consequences of Corollary \ref{c Hakeda Saito 86 for (4) or (3)}. We assume next that $\lambda\in (0,1]$.\smallskip

$(a)$ By hypothesis, $\Phi$ is surjective, then there exists $b\in A$ satisfying $\Phi(b)=0.$ On the other hand, the hypothesis also imply that  $$
\Phi(0)=\Phi(\Delta_{\lambda}(0))=\Phi(\Delta_{\lambda}(b \ 0^*))=\Delta_{\lambda}(\Phi(b) \Phi(0)^*)=\Delta_{\lambda}(0)=0.$$

$(b)$  Let $\Phi: M\to N$  be a bijection satisfying $(h.3)$ (respectively, $(h.4)$). Take $a\in M,$ since $aa^*$ (respectively $a\circ a^*$) is normal we get from the hypothesis that
$$\Phi(aa^*)=\Phi(\Delta_{\lambda}(aa^*))=\Delta_{\lambda}(\Phi(a)\Phi(a)^*)=\Phi(a)\Phi(a)^*,$$ (respectively, $$\Phi(a\circ a^*)=\Phi(\Delta_{\lambda}(a\circ a^*))=\Delta_{\lambda}(\Phi(a)\circ \Phi(a)^*)=\Phi(a)\circ \Phi(a)^*).$$
In particular, $\Phi$ maps positive elements in $M$ to positive elements in $N$. Since $\Phi$ is bijective and its inverse satisfies the same hypothesis, $\Phi$ preserves the set of positive elements in both directions.\smallskip

$(c)$ Let $p\in M$ be a projection. By $(b)$, $\Phi (p)$ is a positive element in $N$ with $\Phi (p) = \Phi (p) \Phi(p)^* = \Phi (p)^2$ (respectively, $\Phi (p) = \Phi (p)\circ \Phi(p)^* = \Phi (p)^2$), which proves that $\Phi (p)$ is a projection. The rest follows from the same arguments.\smallskip

$(d)$ Let us take $p,q\in \hbox{Proj}(M)$ with $p\leq q$. We know from $(c)$ that $\Phi(p)$ and $\Phi(q)$ are projections in $N$. By hypothesis \begin{equation}\label{eq 1 2911} \Phi (p) = \Phi (\Delta_{\lambda}(p)) = \Phi (\Delta_{\lambda}(p q)) = \Delta_{\lambda} (\Phi(p) \Phi(q)),
 \end{equation}(respectively, \begin{equation}\label{eq 2 2911} \Phi (p) = \Phi (\Delta_{\lambda}(p)) = \Phi (\Delta_{\lambda}(p \circ q)) = \Delta_{\lambda} (\Phi(p) \circ \Phi(q))= \Phi(p) \circ \Phi(q)\Big).
  \end{equation}
Combining Lemma \ref{l Chabb 2.3 von Neumann} and \eqref{eq 1 2911} we deduce that $\Phi(p) \Phi(q) =\Phi(q) \Phi(p) = \Phi (p)$ or equivalently $\Phi (p)\leq \Phi (q)$.\smallskip

When $\Phi$ satisfies $(h.4)$, we deduce from \eqref{eq 2 2911} that $ \Phi (p) = \Phi(p) \circ \Phi(q)$, which assures that $\Phi(p) \Phi(q) =\Phi(q) \Phi(p) = \Phi (p)$.\smallskip

$(e)$ By applying $(c)$ and $(d)$ we deduce that $\Phi (\textbf{1})$ is a projection in $N$ and $\Phi(\textbf{1})\geq q$ for every projection $q\in N$. Therefore, $\Phi (\textbf{1}) =\textbf{1}$.\smallskip

$(f)$ Let us take $p,q$ in Proj$(M)$ with $pq =0$. Then $$0=\Phi (0) = \Phi (\Delta_{\lambda} (p q) ) = \Delta_{\lambda} (\Phi (p) \Phi(q)),$$ (respectively, $$0=\Phi (0) = \Phi (\Delta_{\lambda} (p \circ q) ) = \Delta_{\lambda} (\Phi (p) \circ \Phi(q)) = \Phi (p) \circ \Phi(q) \Big).$$

If $\Delta_{\lambda} (\Phi (p) \Phi(q)) =0$ and $\lambda =0$, we get $\Phi (p) \Phi(q)=0$ as desired. If $0<\lambda\leq 1$, we deduce from \eqref{eq kernel Delta lambda are nilpotent} that $\Phi (p) \Phi(q) \Phi (p) \Phi(q)=0.$ This implies that $(\Phi (p) \Phi(q) \Phi (p)) (\Phi (p) \Phi(q) \Phi (p)) =0,$ which assures, via the positivity of $\Phi (p) \Phi(q) \Phi (p)$, that $$0=\Phi (p) \Phi(q) \Phi (p)=\Phi (p) \Phi(q) \Phi(q) \Phi (p)= (\Phi (p) \Phi(q)) (\Phi (p) \Phi(q))^*,$$ and hence $\Phi (p) \Phi(q)=0$.\smallskip

When $\Phi$ satisfies $(h.4)$, we have seen above that $0 = \Phi (p) \circ \Phi(q)$, and consequently, $\Phi (p)  \Phi(q) + \Phi (q)  \Phi(p) =0,$ which gives $\Phi (p)  \Phi(q)  \Phi (p) =0,$ and thus $\Phi (p)  \Phi(q) =0$.\smallskip

$(g)$ is a clear consequence of $(d)$.\smallskip

$(h)$ Let $p_1,\ldots,p_m$ be mutually orthogonal projections in $M$. We known from previous statements that $\Phi(p_1),\ldots, \Phi(p_m)$ and $\displaystyle \Phi \left(\sum_{j=1}^{m} p_j\right)$ are projections in $N$ with $\Phi (p_j)\perp \Phi (p_k)$ for all $j\neq k$, and $\displaystyle \Phi (p_k) \leq \Phi \left(\sum_{j=1}^{m} p_j\right)$ for all $k$. Therefore $\displaystyle \sum_{j=1}^{m} \Phi (p_j)$ is a projection with $\displaystyle \sum_{j=1}^{m} \Phi (p_j)\leq \Phi \left(\sum_{j=1}^{m} p_j\right)$. Applying the same argument to $\Phi^{-1}$, and the projections $\Phi(p_1),\ldots, \Phi(p_m)$, we get $$\sum_{j=1}^{m} p_j = \sum_{j=1}^{m} \Phi^{-1} (\Phi ( p_j))\leq \Phi^{-1} \left(\sum_{j=1}^{m} \Phi (p_j)\right),$$ and by $(d)$ we have $\displaystyle\Phi \left(\sum_{j=1}^{m} p_j\right) \leq \sum_{j=1}^{m} \Phi (p_j)$.\smallskip

$(i)$ Let $p$ and $q$ be projections in $M$ with $p\leq q$. We know from $(h)$ that $$ \Phi (q) = \Phi ((q-p) + p) =  \Phi(q-p) + \Phi (p),$$ and hence $\Phi (q-p) = \Phi (q) -\Phi (p)$. \smallskip

$(j)$ Let $a\in\mathcal{B}(H).$ By hypothesis and $(e)$ we have $$\Delta_{\lambda}(\Phi(a))=\Delta_{\lambda}(\Phi(a) \textbf{1}^*)= \Delta_{\lambda}(\Phi(a) \Phi(\textbf{1})^*)=\Phi(\Delta_{\lambda}(a \textbf{1}^*))=\Phi(\Delta_{\lambda}(a)),$$ (respectively, $\Delta_{\lambda}(\Phi(a))= \Delta_{\lambda}(\Phi(a)\circ \Phi(\textbf{1})^*) = \Phi(\Delta_{\lambda}(a \circ \textbf{1}^*))= \Phi(\Delta_{\lambda}(a))$). \smallskip

$(k)$ Let $a\in M.$ Applying the hypothesis and $(e)$ we get $$\Delta_{\lambda}(\Phi(a)^*)=\Delta_{\lambda}(\Phi(\textbf{1})\Phi(a)^*)
=\Phi(\Delta_{\lambda}(a^*)),$$ (respectively, $\Delta_{\lambda}(\Phi(a)^*) =\Delta_{\lambda}(\Phi(\textbf{1})\circ \Phi(a)^*)
=\Phi(\Delta_{\lambda}(a^*))$).
\end{proof}


The next lemma is probably known, however an explicit reference is not available among our references. 

\begin{lemma}\label{l Aluthge transform equals to 1} Let $a$ be an element in a von Neumann algebra $M$. 
Suppose $\lambda\in [0,1]$. Then $\Delta_{\lambda} (a) = \textbf{1}$ if and only if $a=1$;
\end{lemma}

\begin{proof}
We can clearly assume that $\lambda>0$. The ``If'' implication is clear. Suppose now that $|a|^{\lambda} u |a|^{1-\lambda} = \Delta_{\lambda} (a) = \textbf{1}$, where $a= u |a|$ is the polar decomposition of $a$. We shall first show that $|a|$ is invertible. If $\lambda=\frac12$, then the identity $|a|^{\frac12} u |a|^{\frac12} =  \textbf{1}$ implies that $|a|^{\frac12}$ is left and right invertible, and thus $|a|^{\frac12}$ (and hence $|a|$) is invertible. If $\lambda< \frac12$, we write $|a|^{\lambda} u |a|^{1-2\lambda} |a|^{\lambda} = \textbf{1}$, which guarantees that $|a|^{\lambda}$ (and hence $|a|$) is invertible. We can similarly show that $|a|^{1-\lambda}$ (and hence $|a|$) is invertible when $\lambda>\frac12$.\smallskip

Since $|a|$ is invertible, multiplying the identity $|a|^{\lambda} u |a|^{1-\lambda} = \textbf{1}$ on the left by $|a|^{-\lambda}$, and on the right by $|a|^{-1+\lambda}$ we get $u = |a|^{-\lambda} |a|^{-1+\lambda} = |a|^{-1}$. Therefore $a = u |a| = 1$, which finishes the proof.\smallskip
\end{proof}

In \cite[Lemma 2.3]{ChabbabiMbekhta2017Jordan} F. Chabbabi and M. Mbekhta establish that for a quasi-normal operator $S$ in $\mathcal{B}(H)$ and $\lambda\in (0,1)$ we have $\Delta_{\lambda} (S^*)=S \Rightarrow S^*=S$. By the arguments already applied in previous results we obtain:

\begin{lemma}\label{l q-normal and symmetric} Let $a$ be a quasi-normal element in a von Neumann algebra $M$, and let $\lambda\in [0,1]$. If $\Delta_{\lambda} (a^*)=a $, then $a^*=a$.
\end{lemma}

\begin{proof} The statement for $\lambda =0$ is clear. The statement for $\lambda\in (0,1)$ follows from \cite[Lemma 2.3]{ChabbabiMbekhta2017Jordan}. Actually, the arguments in the proof of \cite[Lemma 2.3]{ChabbabiMbekhta2017Jordan} also cover the case for $\lambda=1$. Indeed, let us assume that $a = u |a|$ and $a^* = u^* |a^*|$ are the polar decompositions of $a$ and $a^*$, respectively. Clearly $u |a| u^* = u a^* = |a^*|$.\smallskip

By hypothesis we have $|a^*| u^* = \Delta_{1} (a^*)=a = u |a|,$ and hence $$|a| = u^* |a^*| u^* = u^* (u |a| u^* ) u^* = |a| u^*  u^*= |a|^* = u u |a|.$$

Since $a$ is quasi-normal we know that $|a| u = u |a|$ and $u^* |a| = |a| u^*$. Therefore, $$a = u |a| = |a| u = u^* u |a| u = u^* |a| u^2 = u^* |a| = |a| u^* = a^*.$$
\end{proof}

\begin{corollary}\label{c hermitian and skew symmetric for h3 and h4} Let $\Phi:M\to N$ be a bijective map between von Neumann algebras satisfying hypothesis $(h.3)$ {\rm(}respectively, $(h.4)${\rm)} in Problem \ref{problems}. Then $\Phi (a)^* = \Phi (a)$ for all $a\in M_{sa}$. Consequently, $\Phi (M_{sa}) = N_{sa}$.
\end{corollary}

\begin{proof} Let us take $a\in M_{sa}$. Applying Proposition \ref{p basic properties}$(k)$ we get $$\Delta_{\lambda} (\Phi(a)^*) = \Phi(\Delta_{\lambda} (a^*)) = \Phi (\Delta_{\lambda} (a)) = \Phi (a).$$ Proposition \ref{p basic properties}$(j)$ guarantees that $\Phi (a)$ is quasi-normal, and Lemma \ref{l q-normal and symmetric} proves that $\Phi (a)^* = \Phi (a)$.
\end{proof}

\section{Maps commuting with the Jordan $^*$-product up to a $\lambda$-Aluthge transform}\label{sec:3}

We shall focus next on maps between general von Neumann algebras satisfying hypothesis $(h.4)$ in Problem \ref{problems}.

\begin{proposition}\label{p preservation of jordan products for positive in h.4} Let $\Phi:M\to N$ be a bijective map between von Neumann algebras satisfying hypothesis $(h.4)$ in Problem \ref{problems}. Let $a,b\in M_{sa}$ such that $\Phi (a),\Phi(b)\in N_{sa}$. Then the identity $\Phi (a\circ b) = \Phi (a)\circ \Phi(b)$ holds. The same identity holds when $a\in M_{sa}$, $b\in i M_{sa}$, $\Phi (a)\in N_{sa}$ and $\Phi(b)\in i N_{sa},$ and when $a,b\in i M_{sa}$ and $\Phi (a), \Phi(b)\in i N_{sa}.$
Consequently, the identity $\Phi (a\circ b) = \Phi (a)\circ \Phi(b)$ holds for all $a,b\in M^+$.
\end{proposition}

\begin{proof} Take $a,b\in M_{sa}$ such that $\Phi (a),\Phi (b)\in N_{sa}$. The Jordan product $\Phi (a)\circ \Phi (b)$ is a hermitian (and hence normal element in $N$). We therefore have $$\Phi (a)\circ \Phi (b) =\Delta_{\lambda} (\Phi (a)\circ \Phi (b))=\Delta_{\lambda} (\Phi (a)\circ \Phi (b)^*) = \Phi \Delta_{\lambda} (a\circ b^*) = \Phi (a \circ b).$$

If $a\in M_{sa}$, $b\in i M_{sa}$, $\Phi (a)\in N_{sa}$ and $\Phi(b)\in i N_{sa},$ since $\Phi (a)\circ \Phi (b)$ is a skew symmetric element (and hence normal), the second statement follows from the same arguments above. The proof of the third statement is very similar.\smallskip

Finally, if $a,b$ are positive elements in $M$, Proposition \ref{p basic properties}$(b)$ proves that $\Phi (a),\Phi (b)$ are positive elements in $N$.  Then the desired identity is a consequence of the first statement.
\end{proof}

\begin{corollary}\label{c hermitian and skew symmetric} Let $\Phi:M\to N$ be a bijective map between von Neumann algebras satisfying hypothesis $(h.4)$ in Problem \ref{problems}. Then the following statements hold: \begin{enumerate}[$(a)$]\item $\Phi (a)^* = \Phi (a)$ for all $a\in M_{sa}$. Consequently, $\Phi(M_{sa}) = N_{sa}$;
\item $\Phi (a \circ b ) = \Phi (a) \circ \Phi (b)$ for all $a,b\in M_{sa}$.
\end{enumerate}
\end{corollary}

\begin{proof} Statement $(a)$ is proved in Corollary \ref{c hermitian and skew symmetric for h3 and h4}, while statement $(b)$ is a consequence of $(a)$ and Proposition \ref{p preservation of jordan products for positive in h.4}.
\end{proof}

Let $A$ be a C$^*$-algebra. We recall that elements $a$ and $b$ in $A_{sa}$ are said to \emph{operator commute} in $A_{sa}$ if the
Jordan multiplication operators $M_a (x) = a\circ x$ and $M_b (x) = b\circ x$ commute, that is, $a$ and $b$ operator
commute if and only if $(a\circ x) \circ b = a\circ (x\circ b)$ for
all $x$ in $A_{sa}$ (or for all $x\in A$). It is known that $a$ and $b$ operator commute if and only if they commute in the usual sense
(see \cite[Proposition 1]{Top}).

\begin{lemma}\label{lemma central elements} Let $\Phi:M\to N$ be a bijective map between von Neumann algebras satisfying hypothesis $(h.4)$ in Problem \ref{problems}. Suppose $a$ and $b$ in $M_{sa}$ (operator) commute in $M$. Then $\Phi(a)$ and $\Phi(b)$ (operator) commute in $N_{sa}$.  Consequently, $\Phi$ maps the center of $M$ to the center of $N$.
\end{lemma}

\begin{proof} $a$ and $b$ in $M_{sa}$ (operator) commute in $M$ if and only if $(a\circ x) \circ b = a\circ (x\circ b)$ for
all $x$ in $M_{sa}$. By Corollary \ref{c hermitian and skew symmetric} we have $$(\Phi(a)\circ \Phi(x)) \circ \Phi(b) = \Phi(a\circ x) \circ \Phi(b) = \Phi\left((a\circ x) \circ b\right) = \Phi\left((b\circ x) \circ a\right)$$ $$= \Phi(b\circ x) \circ \Phi(a) = \Phi(a)\circ (\Phi(x)\circ \Phi(b)),$$ for all $x\in M_{sa}$, which assures that $\Phi(a)$ and $\Phi(b)$ (operator) commute in $N_{sa}$.
\end{proof}

Inspired by techniques developed by J. Hakeda and K. Sait\^{o} in \cite{Hak86a, Hak86b, HakSaito86} we establish our next result.

\begin{proposition}\label{p consequences preservation of jordan products for positive in h.4} Let $\Phi:M\to N$ be a bijective map between von Neumann algebras satisfying hypothesis $(h.4)$ in Problem \ref{problems}. Then the following statements hold:\begin{enumerate}[$(a)$]
\item Let $p_1,\ldots,p_m$ be mutually orthogonal projections in $M$, and let $\alpha_1,\ldots, \alpha_m$ be elements in $\mathbb{R}$. Then the identity $$\displaystyle \Phi\left(\sum_{j=1}^{m} \alpha_j p_j\right)=\sum_{j=1}^{m} \Phi(\alpha_j p_j)$$ holds;
\item $\Phi (-\textbf{1}) = -\Phi(\textbf{1})=-\textbf{1}$;
\item $\Phi (-x) = -\Phi (x)$, for every hermitian (or skew symmetric) element $x$ in $ M$;
\item $\Phi (a)^* = - \Phi (a)$ for all $a\in i M_{sa}$;
\item $\Phi (a \circ b ) = \Phi (a) \circ \Phi (b)$ for all $a\in M_{sa},$ $b\in i M_{sa}$;
\item $\Phi (a \circ b ) = \Phi (a) \circ \Phi (b)$ for all $a,b\in i M_{sa}$;
\item $\Phi (-x) = -\Phi (x)$, for every skew symmetric element $x$ in $M$;
\item For each projection $p$ in $M$ we have $\Phi (-p) = -\Phi(p)$ and $\Phi (2 p -\textbf{1}) = 2 \Phi (p)-\textbf{1}$;
\item The identity $\Phi (p b p) = \Phi (p) \Phi(b) \Phi (p)$ holds for every $b\in M_{sa}\cup i M_{sa}$ and every projection $p$ in $M$;
\item Suppose $M$ is a von Neumann algebra which has no abelian direct summand. Then $\Phi|_{\mathbb{R} \textbf{1}}$ is additive;
\end{enumerate}
\end{proposition}

\begin{proof} As in the proof of Proposition \ref{p basic properties}, the statements in the case $\lambda=0$ follow from Corollary \ref{c Hakeda Saito 86 for (4) or (3)}.\smallskip

$(a)$ Applying Proposition \ref{p basic properties}$(c)$, $(f)$ and $(h)$, we know that $\Phi(p_1),$ $\ldots,$ $\Phi(p_m)$ are mutually orthogonal projections in $N,$ and $$\displaystyle \Phi\left(\sum_{j=1}^{m} p_j\right)=\sum_{j=1}^{m} \Phi(p_j)$$ is another projection in $N$. Now, applying Corollary \ref{c hermitian and skew symmetric} twice we get $$\Phi\left(\sum_{j=1}^{m} \alpha_j p_j\right)= \Phi\left( \left(\sum_{j=1}^{m} \alpha_j p_j\right) \circ \left(\sum_{k=1}^{m} p_k\right)\right)= \Phi\left(\sum_{j=1}^{m} \alpha_j p_j\right) \circ \Phi\left(\sum_{k=1}^{m} p_k\right)$$ $$ = \Phi\left(\sum_{j=1}^{m} \alpha_j p_j\right) \circ \left(\sum_{k=1}^{m} \Phi(p_k) \right) = \sum_{k=1}^{m} \Phi\left(\sum_{j=1}^{m} \alpha_j p_j\right) \circ \Phi(p_k)  $$ $$ =  \sum_{k=1}^{m} \Phi\left( \left(\sum_{j=1}^{m} \alpha_j p_j\right) \circ p_k\right) = \sum_{k=1}^{m} \Phi(\alpha_k p_k).$$

$(b)$ Corollary \ref{c hermitian and skew symmetric} proves that $\Phi (- \textbf{1})$ is a hermitian element and $$ \Phi (- \textbf{1}) \circ \Phi (- \textbf{1}) = \Phi \left( (-\textbf{1})\circ (- \textbf{1})\right) =\Phi (\textbf{1})=\hbox{(Proposition \ref{p basic properties}$(e)$)} =\textbf{1}.$$ It is well known from spectral theory that in this case we have $\Phi (- \textbf{1}) = q - (\textbf{1}-q)$ for a unique projection $q$ in $N$. By applying Proposition \ref{p basic properties}$(d)$, $(f)$ and $(h)$ we find a projection $e$ in $M$ satisfying $\Phi (e) = \textbf{1}-q$ and $\Phi (\textbf{1}-e) = q$. Now, Corollary \ref{c hermitian and skew symmetric} implies that $$-\Phi (e) =- (\textbf{1}-q) = (\textbf{1}-q)\circ ( q - (\textbf{1}-q) ) = \Phi(e)\circ \Phi (- \textbf{1}) =  \Phi(e \circ(- \textbf{1}) ) = \Phi (-e). $$ Therefore, by $(a)$, $$\Phi ( (\textbf{1}-e) - e) = \Phi ((\textbf{1}-e)) + \Phi (-e) = \Phi ((\textbf{1}-e)) - \Phi (e) = q - (\textbf{1}-q) =\Phi (-\textbf{1}).$$ We deduce from the injectivity of $\Phi$ that $-\textbf{1} = (\textbf{1}-e) - e$, and hence $e = \textbf{1}$ and $-\textbf{1}= -\Phi (\textbf{1})= -\Phi(e) = \Phi (-e) = \Phi (-\textbf{1}).$ \smallskip

$(c)$ Let $x$ be a hermitian element in $M$. By Corollary \ref{c hermitian and skew symmetric} and $(b)$ we get $- \Phi(x) =  \Phi (x) \circ (-\textbf{1})= \Phi (x) \circ \Phi (-\textbf{1}) =\Phi (x \circ (-\textbf{1}) ) = \Phi (-x)$.\smallskip

$(d)$ Suppose $a\in i M$. Proposition \ref{p basic properties}$(k)$ and $(c)$ assure that  $$\Delta_{\lambda} (\Phi(a)^*) = \Phi(\Delta_{\lambda} (a^*)) = \Phi (\Delta_{\lambda} (- a)) = \Phi (-a) = - \Phi (a).$$ Proposition \ref{p basic properties}$(j)$ guarantees that $\Phi (a)$ is quasi-normal, and Lemma \ref{l q-normal and symmetric} proves that $\Phi (a)^* = \Phi (a)$.\smallskip

Statements $(e)$ and $(f)$ are clear consequences of $(d)$ and  Proposition \ref{p preservation of jordan products for positive in h.4}, while $(g)$ follows from $(d)$ and $(e)$.\smallskip

$(h)$ Let us take a projection $p$ in $M$. By $(a)$ and $(c)$ we have $$ \Phi (2 p- \textbf{1})  = \Phi \left( p - (\textbf{1}-p)\right) = \Phi (p ) + \Phi (- (\textbf{1}-p) ) $$ $$= \Phi (p ) - \Phi (\textbf{1}-p) =  \hbox{(Proposition \ref{p basic properties}$(i)$)}=  2 \Phi (p)-\textbf{1}. $$

$(i)$ Let us fix a projection $p$ in $M$ and $b\in M_{sa}\cup i M_{sa}$. The elements $(2p -  \textbf{1} )$ and $p$ are hermitian. By $(h)$ we know that $\Phi (2 p- \textbf{1}) = 2 \Phi (p)-\textbf{1}$. Corollary \ref{c hermitian and skew symmetric} and $(e)$ assert that $$\Phi ( p b p ) = \Phi (((2 p- \textbf{1})\circ b) \circ p ) =  \Phi (((2 p- \textbf{1})\circ b)) \circ \Phi (p ) $$ $$= \left( (\Phi (2 p- \textbf{1}) \circ \Phi(b))\right) \circ \Phi (p ) =  \left( (2 \Phi (p)-\textbf{1}) \circ \Phi(b))\right) \circ \Phi (p ) = \Phi (p) \Phi (b) \Phi (p).$$

$(j)$ Let $M$ be a von Neumann algebra which has no abelian direct summand.  Arguing as in \cite[proof of Corollary 2.7]{Hak86a} or in \cite[proof of COROLLARY]{HakSaito86}, we can find a family $\{p_{k} : k\in I\}$  of central orthogonal projections in $M$ such that $\displaystyle \sum_{k\in I} p_{k}=\textbf{1}$, there exists $k_0\in I$ such that  $M p_{k_0}$ has no finite type I direct summand, and $M p_{k}$ is homogeneous of type $I_{n_{k}}$ $(n_{k}\geqq 2)$  for all  $k\neq k_0$. Proposition \ref{p basic properties}$(c)$ and $(f)$ and Lemma \ref{lemma central elements} imply that  $\{\Phi(p_{k}) : k\in I\}$ is a family of central orthogonal projections in $N$. Clearly, $\displaystyle \sum_{k\in I} \Phi(p_k)$ is a central projection in $N$, where the series converges with respect to the weak$^*$ topology of $N$ (cf. \cite[Definition 1.13.4]{S}). If $\displaystyle \sum_{k\in I} \Phi(p_k)\neq \textbf{1}$, again by Proposition \ref{p basic properties}$(c)$ and $(f)$ and Corollary \ref{c hermitian and skew symmetric}, we can find a central projection $p\in M$ such that $\Phi (p) =  \textbf{1}- \displaystyle \sum_{k\in I} \Phi(p_k)$. Since the product in $N$ is separately weak$^*$-continuous (see \cite[Theorem 1.7.8]{S}), we also get from  Corollary \ref{c hermitian and skew symmetric} that $$0=\Phi(p) \Phi(p_k) = \Phi(p)\circ \Phi(p_k) = \Phi(p\circ p_k)= \Phi(p p_k),$$ for all $i\in I$. The injectivity of $\Phi$ assures that $p p_k=0$ for all $k\in I$, and hence $p=0$ and $\textbf{1}= \displaystyle \sum_{k\in I} \Phi(p_k)$.\smallskip

For each $k\in I$, let $\{u^k_{ij}: i,j=1,\ldots,n_k \}$ be a \emph{system of $n \times n$ matrix units} in $M p_k$, where $n_k$ is an integer greater than or equal to $2$. Let us fix $\alpha,\beta\in \mathbb{R}$. Taking $i\neq j$ in $\{1,\ldots,n_k\}$, the elements $p^k = \frac12 (u^k_{ii} + (u^k_{ij})^*)(u^k_{ii} + u^k_{ij})$ and $q^k = \frac12 (u^k_{ii} - (u^k_{ij})^*)(u^k_{ii} - u^k_{ij})$ are orthogonal projections in $M$. Therefore, by applying Corollary \ref{c hermitian and skew symmetric} we get
$$  \Phi \left( (\alpha+\beta) \textbf{1} \right)\circ \Phi (u^k_{ii}) = \Phi \left( (\alpha+\beta) \textbf{1} \circ u^k_{ii}\right)= \Phi \left( u^k_{ii} ((\alpha+\beta) \textbf{1}) u^k_{ii}\right) $$ $$ = \Phi \left( u^k_{ii} ( 2 \alpha p^k + 2 \beta q^k)) u^k_{ii}\right) =\hbox{(by $(i)$)}= \Phi ( u^k_{ii}) \Phi \left( 2 \alpha p^k + 2 \beta q^k\right) \Phi (u^k_{ii}) $$ $$ =\hbox{(by $(a)$ above)}=  \Phi ( u^k_{ii}) \left( \Phi \left( 2 \alpha p^k \right) + \Phi \left(2 \beta q^k\right) \right) \Phi (u^k_{ii}) $$
$$= \Phi ( u^k_{ii}) \Phi \left( 2 \alpha p^k \right)  \Phi (u^k_{ii}) + \Phi ( u^k_{ii}) \Phi \left(2 \beta q^k\right) \Phi (u^k_{ii}) $$ $$=\hbox{(by $(i)$)}= \Phi (  2 \alpha u^k_{ii} p^k u^k_{ii}) + \Phi ( 2 \beta u^k_{ii} q^k u^k_{ii}) = \Phi (  \alpha u^k_{ii}) + \Phi ( \beta u^k_{ii})  $$ $$ = \Phi (  ( \alpha  \textbf{1}) \circ u^k_{ii}) + \Phi ( (\beta  \textbf{1}) \circ u^k_{ii}) = \Phi ( \alpha  \textbf{1}) \circ  \Phi ( u^k_{ii}) + \Phi (\beta  \textbf{1}) \circ  \Phi (u^k_{ii}),$$ which assures that
$$ \Phi \left( (\alpha+\beta) \textbf{1} \right) = \Phi \left( (\alpha+\beta) \textbf{1} \right) \Phi(\textbf{1}) = \Phi \left( (\alpha+\beta) \textbf{1} \right) \Phi\left(  p_{k} \right)  $$
$$= \Phi \left( (\alpha+\beta) \textbf{1} \right) \left( \sum_{k\in I} \Phi(p_{k}) \right) =  \sum_{k\in I} \Phi \left( (\alpha+\beta) \textbf{1} \right) \Phi\left(  p_{k} \right) $$
$$=\sum_{k\in I} \Phi \left( (\alpha+\beta) \textbf{1} \right)\circ \Phi \left(\sum_{i=1}^{n_k} u^k_{ii} \right)=\hbox{(by Proposition \ref{p basic properties}$(f)$)} $$ $$= \sum_{k\in I} \Phi \left( (\alpha+\beta) \textbf{1} \right)\circ  \left(\sum_{i=1}^{n_k} \Phi (u^k_{ii})  \right)  = \sum_{k\in I}   \sum_{i=1}^{n_k}  \Phi \left( (\alpha+\beta) \textbf{1} \right)\circ \Phi (u^k_{ii})  $$ $$= \sum_{k\in I}  \sum_{i=1}^{n_k} \left( \Phi ( \alpha  \textbf{1}) \circ  \Phi ( u^k_{ii}) + \Phi (\beta  \textbf{1}) \circ  \Phi (u^k_{ii})\right) $$ $$=   \sum_{k\in I}  \Phi ( \alpha  \textbf{1}) \circ  \left(\sum_{i=1}^{n_k}\Phi ( u^k_{ii}) \right)+ \sum_{k\in I} \Phi (\beta  \textbf{1}) \circ  \left(\sum_{i=1}^{n_k} \Phi (u^k_{ii}) \right)  $$ $$=   \sum_{k\in I}  \Phi ( \alpha  \textbf{1}) \circ  \Phi \left(\sum_{i=1}^{n_k} u^k_{ii} \right)+ \sum_{k\in I} \Phi (\beta  \textbf{1}) \circ  \Phi\left(\sum_{i=1}^{n_k}  u^k_{ii} \right) $$ $$=   \sum_{k\in I}  \Phi ( \alpha  \textbf{1}) \circ  \Phi \left(p_k \right)+ \sum_{k\in I} \Phi (\beta  \textbf{1}) \circ  \Phi\left(p_k \right) $$ $$=    \Phi ( \alpha  \textbf{1}) \circ  \left( \sum_{k\in I} \Phi \left(p_k \right)\right) + \Phi (\beta  \textbf{1}) \circ  \left( \sum_{k\in I} \Phi\left(p_k \right)\right)=\Phi ( \alpha  \textbf{1}) + \Phi ( \beta  \textbf{1}).$$
\end{proof}

Our next result, which can be considered a consequence of Kantorovic's Theorem  (see \cite[Theorem 1.7]{AlBurk1985}), will play a key role in our arguments.

\begin{lemma}\label{l additivity on hermitian gives lienarity on hermitian} Let $A$ and $B$ be C$^*$-algebras. Suppose $F: A_{sa}\to B_{sa}$ is a mapping satisfying the following assumptions:
\begin{enumerate}[$(1)$]
\item $F(0)=0$;
\item $F$ is additive;
\item $F$ is Jordan multiplicative, i.e, $F(x \circ y) = F(x) \circ {F(y)}$ for all $x,y\in A_{sa}$.
\end{enumerate} Then $F$ is linear.
\end{lemma}

\begin{proof} We shall first prove that $F(-z) =- F(z),$ for all $z\in A_{sa}$. Indeed, by the additivity of $F$ we get $0 = F(0) = F(z-z) = F (z) + F(-z),$ and hence $F(-z) = -F(z)$.\smallskip

We shall next show that $h(A^+)\subseteq B^+$. Namely, if $x\in A^+$ we can pick $y\in A^+$ satisfying $y^2 = x$. Since $F$ is Jordan multiplicative, we have $F(x ) = F(y^2) = F(y)^2,$ which proves the desired statement.\smallskip

Since $F$ is additive and $F(-z) = - F(z)$ for all $z\in A_{sa}$, we deduce by induction that $F (r z) = r F(z)$ for all $z\in A_{sa}$, and all $r\in \mathbb{Q}$. \smallskip

We shall next show that $F(\alpha x) = \alpha F(x)$ for all $x\in A^+$ and for all real $\alpha$. We apply an argument which is already in the proof of Kantorovic's Theorem  (see \cite[Theorem 1.7]{AlBurk1985}). For each real $\alpha$, we can find sequences $(r_n)$ and $(t_n)$ in $\mathbb{Q}$ such that $r_n\leq \alpha\leq t_n,$ for every $n$, $(r_n)\to \alpha$ and $(t_n)\to \alpha$. Since $F$ is positive, additive, and $\mathbb{Q}$-linear we have $$ r_n F(x) = F(r_n x) \leq F(\alpha x) \leq F(t_n x) = t_n F(x),$$ for all natural $n$. Taking limits in $n$ we get $F(\alpha x ) = \alpha F(x)$. Now, given $\alpha\in \mathbb{R}$ and $z\in A_{sa}$ we write $z = z^+ - z^-$ with $z^+$, $z^-$ in $A^+$, and hence $$F(\alpha z) = F(\alpha z^+ - \alpha z^-) =  F(\alpha z^+) + F(- \alpha z^-)= \alpha  F( z^+)  - \alpha F(z^-) = \alpha F(z).$$\end{proof}

\begin{corollary}\label{c linearity on R1} Let $\Phi:M\to N$ be a bijective map between von Neumann algebras satisfying hypothesis $(h.4)$ in Problem \ref{problems}. Suppose $M$ is a von Neumann algebra which has no abelian direct summand. Then $\Phi( \alpha \textbf{1}) = \alpha \textbf{1}$ for all $\alpha\in \mathbb{R}$.
\end{corollary}

\begin{proof} Corollary \ref{c hermitian and skew symmetric for h3 and h4} assures that $\Phi (M_{sa})= N_{sa}$ and hence $\Phi|_{\mathbb{R} \textbf{1}} : \mathbb{R} \textbf{1} \to N_{sa}$. Proposition \ref{p basic properties}$(a)$ implies that $\Phi (0) =0$. Corollary \ref{c hermitian and skew symmetric} implies that $\Phi|_{\mathbb{R} \textbf{1}}$ is Jordan multiplicative, and Proposition \ref{p consequences preservation of jordan products for positive in h.4}$(j)$ implies that $\Phi|_{\mathbb{R} \textbf{1}}$ is additive. We deduce from Lemma \ref{l additivity on hermitian gives lienarity on hermitian} that $\Phi|_{\mathbb{R} \textbf{1}}$ is linear. Since $\Phi(\textbf{1}) = \textbf{1}$, it follows that $\Phi( \alpha \textbf{1}) = \alpha \textbf{1}$ for all $\alpha\in \mathbb{R}$.
\end{proof}

Our next result can be deduced via arguments developed by J. Hakeda and K. Sait\^{o} in \cite[Lemmas 5 to 8]{HakSaito86}. We include here a complete proof for completeness reasons.

\begin{proposition}\label{p additivity on hermitian} Let $\Phi:M\to N$ be a bijective map between von Neumann algebras satisfying hypothesis $(h.4)$ in Problem \ref{problems}. Suppose $M$ is a von Neumann algebra which has no abelian direct summand. Then $\Phi$ is additive on hermitian elements.
\end{proposition}

\begin{proof} Taking a closer look at the proof of Proposition \ref{p consequences preservation of jordan products for positive in h.4}$(j)$ we deduce that for each family $\{p_{k} : k\in I\}$  of mutually orthogonal projections in $M$ such that $\displaystyle \sum_{k\in I} p_{k}=\textbf{1}$ we have $ \displaystyle\textbf{1}=\Phi\left( \sum_{k\in I} p_{k} \right) = \sum_{k\in I} \Phi(p_k)$. In particular for each $x$ in $M_{sa}$ we have \begin{equation}\label{eq L5 HakSa} \Phi (x) = \sum_{k\in I} \Phi(p_{k}) \Phi (x) \Phi(p_k) + \sum_{k\neq j\in I} \Phi(p_{k}) \Phi (x) \Phi(p_j).
\end{equation}

Let us pick a real $\alpha$ and a projection $p\in M$. Combining Corollary \ref{c linearity on R1} and Corollary \ref{c hermitian and skew symmetric} we get \begin{equation}\label{eq linearity on projections} \Phi (\alpha p) = \Phi (\alpha p \circ \textbf{1}) = \Phi (p)\circ \Phi (\alpha \textbf{1}) = \alpha \Phi (p), \ \ \forall p\in \hbox{Proj} (M).
\end{equation}

Let us consider the triple product $\{.,.,\}$ defined by $\{a,b,c\}=\frac12 ( a b^* c + c b^* a)$. Since for $x\in M_{sa}$ and $p\in \hbox{Proj}(M)$ the identity $(2p\circ x) \circ (\textbf{1}-p) =\{p,x,\textbf{1}-p\}$ holds, we deduce, from Corollary \ref{c hermitian and skew symmetric}, \eqref{eq linearity on projections} and Proposition \ref{p basic properties}$(i)$, that \begin{equation}\label{eq preserves triple products of permitian and projections} \Phi \{p,x,\textbf{1}-p \} = \{\Phi(p), \Phi(x),\textbf{1}-\Phi(p)\},
\end{equation} $$ \hbox{ and, } \Phi (p x p ) = \Phi(p) \Phi(x) \Phi(p),$$ for all $x\in M_{sa},$ and $p\in \hbox{Proj} (M).$\smallskip

We claim that \begin{equation}\label{eq L 6 HakSa} \Phi(\alpha \textbf{1} + \beta p +\gamma q ) = \Phi(\alpha \textbf{1}) + \Phi(\beta p) +\Phi(\gamma q )= \alpha \Phi( \textbf{1}) +\beta \Phi( p) +\gamma \Phi( q ),
\end{equation} for all $\alpha,\beta, \gamma \in \mathbb{R}$, and all $p,q\in \hbox{Proj} (M).$ Indeed, set $x= \alpha \textbf{1} + \beta p +\gamma q$. By \eqref{eq preserves triple products of permitian and projections} we have $$\Phi (p) \Phi(x) \Phi(p) = \Phi (p x p) = \Phi \Big(p \Big((\alpha + \beta)  \textbf{1} +\gamma q\Big) p\Big) $$ $$= \Phi \Big(p \Big((\alpha + \beta +\gamma) q + (\alpha+\beta) (  \textbf{1}-q) \Big) p\Big)$$ $$= \Phi (p)  \Phi \Big((\alpha + \beta +\gamma) q + (\alpha+\beta) (  \textbf{1}-q) \Big) \Phi(p) =\hbox{(by Proposition \ref{p consequences preservation of jordan products for positive in h.4}$(a)$)}$$ $$= \Phi (p)  \Big( \Phi ((\alpha + \beta +\gamma) q)  + \Phi ((\alpha+\beta) (  \textbf{1}-q)) \Big) \Phi(p) =\hbox{(by \eqref{eq linearity on projections})}$$ $$ = \Phi (p)  \Big( (\alpha + \beta +\gamma) \Phi ( q)  + (\alpha+\beta) \Phi (\textbf{1}-q) \Big) \Phi(p)$$
$$= \Phi (p)  \Big( \gamma \Phi ( q)  + (\alpha+\beta) \Phi (\textbf{1}) \Big) \Phi(p) =\hbox{(by \eqref{eq linearity on projections} and \eqref{eq preserves triple products of permitian and projections})} $$ $$= \Phi (p)  \Big( \Phi(\alpha \textbf{1}) + \Phi(\beta p) +\Phi(\gamma q ) \Big) \Phi(p) ,$$
$$\Phi(\textbf{1}-p) \Phi (x) \Phi (\textbf{1}-p) = \Phi ((\textbf{1}-p) x  (\textbf{1}-p)) =\Phi ((\textbf{1}-p) (\alpha \textbf{1}+\gamma q)  (\textbf{1}-p)) $$ $$= \Phi(\textbf{1}-p) \Phi (\alpha \textbf{1}+\gamma q) \Phi (\textbf{1}-p) = \Phi(\textbf{1}-p) \Phi ((\alpha +\gamma) q + \alpha (\textbf{1}-q)) \Phi (\textbf{1}-p)$$ $$= \Phi(\textbf{1}-p) \left((\alpha +\gamma) \Phi (q) + \alpha \Phi(\textbf{1}-q)\right) \Phi (\textbf{1}-p)$$ $$= \Phi(\textbf{1}-p) \left( \alpha \textbf{1}  +\gamma \Phi (q) \right) \Phi (\textbf{1}-p) = \Phi(\textbf{1}-p) \Big( \Phi(\alpha \textbf{1}) + \Phi(\beta p) +\Phi(\gamma q ) \Big) \Phi (\textbf{1}-p),$$
and by \eqref{eq preserves triple products of permitian and projections} and Proposition \ref{p consequences preservation of jordan products for positive in h.4}$(a)$ and $(h)$
$$\{\Phi(\textbf{1}-p), \Phi (x), \Phi (p)\} = \Phi \{\textbf{1}-p, x, p\} = \Phi \{\textbf{1}-p, \gamma q, p\} $$ $$ = \{\Phi(\textbf{1}-p),  \Phi( \gamma q), \Phi(p)\} = \{\textbf{1} -\Phi(p),  \Phi( \gamma q), \Phi(p)\}$$ $$ = \{\textbf{1} -\Phi(p), \alpha \Phi( \textbf{1}) +\beta  \Phi(p) +\gamma \Phi( q ), \Phi(p)\} $$ $$= \{\Phi(\textbf{1}-p), \Phi(\alpha \textbf{1}) + \Phi(\beta p) +\Phi(\gamma q ), \Phi(p)\}.$$

The last three identities together prove that $$\Phi(x) =  \Phi (p) \Phi(x) \Phi(p)  + \Phi (\textbf{1}-p) \Phi(x) \Phi(\textbf{1}-p) + 2 \{\Phi(\textbf{1}-p), \Phi (x), \Phi (p)\} $$ $$= \Phi (p)  \Big( \Phi(\alpha \textbf{1}) + \Phi(\beta p) +\Phi(\gamma q ) \Big) \Phi(p) $$ $$+  \Phi(\textbf{1}-p) \Big( \Phi(\alpha \textbf{1}) + \Phi(\beta p) +\Phi(\gamma q ) \Big) \Phi (\textbf{1}-p) $$ $$+ 2 \{\Phi(\textbf{1}-p),  \Phi(\alpha \textbf{1}) + \Phi(\beta p) +\Phi(\gamma q ), \Phi(p)\} $$ $$=  \Phi(\alpha \textbf{1}) + \Phi(\beta p) +\Phi(\gamma q ) = \alpha \Phi( \textbf{1}) +\beta \Phi( p) +\gamma \Phi( q ) ,$$ which finishes the proof of \eqref{eq L 6 HakSa}.\smallskip

We shall next show that \begin{equation}\label{equ L7 HakSa} \Phi (\alpha u + \beta v ) = \Phi (\alpha u) + \Phi(\beta v) = \alpha \Phi  (u) + \beta \Phi  (v),
\end{equation} for all $\alpha,\beta\in \mathbb{R}$, $u,v$ symmetries in $M$ (i.e. $u,v\in M_{sa}$ with $u^2= v^2 =\textbf{1}$). Namely, by spectral theory, the elements $p = \frac12 (\textbf{1}+u)$ and $q = \frac12 (\textbf{1}+v)$ are projections in $M$ and we can write $$ \Phi (\alpha u + \beta v )= \Phi (2 \alpha p + 2 \beta q - (\alpha+\beta)\textbf{1}) =\hbox{(by \eqref{eq L 6 HakSa})}$$ $$= 2 \alpha \Phi (p) + 2 \beta \Phi (q) - (\alpha+\beta) \textbf{1} = \alpha (2 \Phi (p) -\textbf{1}) + \beta (2 \Phi (q) -\textbf{1}) $$ $$=\hbox{(by Proposition \ref{p consequences preservation of jordan products for positive in h.4}$(h)$)} = \alpha \Phi  (2 p -\textbf{1}) + \beta \Phi  (2 q -\textbf{1}) = \alpha \Phi  (u) + \beta \Phi  (v)$$ $$ = \alpha \Phi  ( p -(\textbf{1}-p)) + \beta \Phi  ( q -(\textbf{1}-q)) =\hbox{(by Proposition \ref{p consequences preservation of jordan products for positive in h.4}$(a)$)}$$ $$ = \Phi  (  \alpha p -  \alpha(\textbf{1}-p)) + \Phi  ( \beta q - \beta(\textbf{1}-q)) = \Phi (\alpha u) + \Phi(\beta v).$$\smallskip

Finally, let us pick $x, y\in M_{sa}$. Arguing as in the proof of Proposition \ref{p consequences preservation of jordan products for positive in h.4}$(j)$, since $M$ contains no abelian direct summand, we can find a family $\{p_{k} : k\in I\}$  of central orthogonal projections in $M$ such that $\displaystyle \sum_{k\in I} p_{k}=\textbf{1}$, there exists $k_0\in I$ such that  $M p_{k_0}$ has no finite type I direct summand, and $M p_{k}$ is homogeneous of type $I_{n_{k}}$ $(\mathbb{N}\ni n_{k}\geqq 2)$  for all  $k\neq k_0$. For each $k\in I$, let $\{u^k_{ij}: i,j=1,\ldots,n_k \}$ be a \emph{system of $n \times n$ matrix units} in $M p_k$, where $n_k$ is an integer greater than or equal to $2$.\smallskip

Let us set $x^k = p_k x,$ $y^k = y p_k$, $x_1^k = \gamma_k^{-1} x^k$ and $y_1^k = \gamma_k^{-1} y^k$, where $\gamma_k = \|x_1^k\|+\|y_1^k\|$. It is explicitly shown in the proof of \cite[Lemma 8]{HakSaito86} (see also \cite[Lemma 1]{Hak86c} and \cite[Lemma 3.5]{Hak86b}) that there exist symmetries $v^k_i$, $w^k_i,$ $v^k_{ij}$ and $w^k_{ij}$ in $M p_k$ satisfying \begin{equation}\label{eq normalizzed elements replaced by symmetries}  u_{ii}^k x_1^k u^k_{ii} = u_{ii}^k v_i^k u^k_{ii}, \  \{u^k_{ii}, x_1^k, u_{jj}^k\} = \{u^k_{ii}, v^k_{ij}, u_{jj}^k\},\end{equation}
$$ u_{ii}^k y_1^k u^k_{ii} = u_{ii}^k w_i^k u^k_{ii}, \hbox{ and } \{u^k_{ii}, y_1^k, u_{jj}^k\} = \{u^k_{ii}, w^k_{ij}, u_{jj}^k\}, \ \forall i\neq j.$$ Applying the above properties we get $$\Phi (x+y ) = \sum_{k\in I} \Phi (x + y ) \Phi(p_k)=  \sum_{k\in I} \Phi(p_k) \Phi (x + y ) \Phi(p_k)  $$ $$=  \sum_{k\in I} \sum_{i,j=1}^{n_k} \Phi(u_{ii}^k) \Phi (x + y ) \Phi(u_{jj}^k) $$ $$=  \sum_{k\in I} \left( \sum_{i=1}^{n_k} \Phi(u_{ii}^k) \Phi (x + y ) \Phi(u_{ii}^k)  + 2 \sum_{i<j}^{n_k} \{\Phi(u_{ii}^k), \Phi (x + y ), \Phi(u_{jj}^k)\} \right)$$ $$=\hbox{(by \eqref{eq preserves triple products of permitian and projections})}=  \sum_{k\in I} \left( \sum_{i=1}^{n_k} \Phi(u_{ii}^k (x^k + y^k ) u_{ii}^k)  + 2 \sum_{i<j}^{n_k} \Phi(\{u_{ii}^k, x^k + y^k , u_{jj}^k\}) \right)$$ $$=  \sum_{k\in I} \Phi(\gamma_k p_k) \left( \sum_{i=1}^{n_k} \Phi(u_{ii}^k (x^k_1 + y^k_1 ) u_{ii}^k)  + 2 \sum_{i<j}^{n_k} \Phi(\{u_{ii}^k, x^k_1 + y^k_1 , u_{jj}^k\}) \right)$$
by \eqref{eq normalizzed elements replaced by symmetries}
$$=  \sum_{k\in I} \Phi(\gamma_k p_k) \left( \sum_{i=1}^{n_k} \Phi(u_{ii}^k (v_i^k + w_i^k ) u_{ii}^k)  + 2 \sum_{i<j}^{n_k} \Phi(\{u_{ii}^k, v^k_{ij} + w^k_{ij} , u_{jj}^k\}) \right)$$
$$=  \sum_{k\in I} \Phi(\gamma_k p_k) \left( \sum_{i=1}^{n_k} \Phi(u_{ii}^k) \Phi (v_i^k + w_i^k ) \Phi(u_{ii}^k)  \right. $$ $$\left.+ 2 \sum_{i<j}^{n_k} \{\Phi(u_{ii}^k), \Phi(v^k_{ij} + w^k_{ij}) ,\Phi( u_{jj}^k) \} \right) = \hbox{(by \eqref{equ L7 HakSa})}$$
$$ =  \sum_{k\in I} \Phi(\gamma_k p_k) \left( \sum_{i=1}^{n_k} \Phi(u_{ii}^k) \Phi (v_i^k ) \Phi(u_{ii}^k) + \sum_{i=1}^{n_k} \Phi(u_{ii}^k) \Phi (w_i^k ) \Phi(u_{ii}^k) \right.$$ $$\left.+ 2 \sum_{i<j}^{n_k} \{\Phi(u_{ii}^k), \Phi(v^k_{ij}) ,\Phi( u_{jj}^k) \} + 2 \sum_{i<j}^{n_k} \{\Phi(u_{ii}^k), \Phi( w^k_{ij}) ,\Phi( u_{jj}^k) \} \right)$$
$$ = \sum_{k\in I} \Phi(\gamma_k p_k) \left( \sum_{i=1}^{n_k} \Phi(u_{ii}^k v_i^k u_{ii}^k) + 2 \sum_{i<j}^{n_k} \Phi(\{ u_{ii}^k, v^k_{ij} , u_{jj}^k \}) \right. $$ $$ \left.+ \sum_{i=1}^{n_k} \Phi(u_{ii}^k w_i^k u_{ii}^k) + 2 \sum_{i<j}^{n_k} \Phi(\{ u_{ii}^k, w^k_{ij} , u_{jj}^k \}) \right)$$ $$ = \sum_{k\in I} \Phi(\gamma_k p_k) \left( \sum_{i=1}^{n_k} \Phi(u_{ii}^k x_1^k u_{ii}^k) + 2 \sum_{i<j}^{n_k} \Phi(\{ u_{ii}^k, x^k_{1} , u_{jj}^k \}) \right. $$ $$ \left.+ \sum_{i=1}^{n_k} \Phi(u_{ii}^k y_1^k u_{ii}^k) + 2 \sum_{i<j}^{n_k} \Phi(\{ u_{ii}^k, y^k_{1} , u_{jj}^k \}) \right)$$
$$ = \sum_{k\in I} \left( \sum_{i=1}^{n_k} \Phi(u_{ii}^k x u_{ii}^k) + 2 \sum_{i<j}^{n_k} \Phi(\{ u_{ii}^k, x , u_{jj}^k \}) \right. $$ $$ \left.+ \sum_{i=1}^{n_k} \Phi(u_{ii}^k y u_{ii}^k) + 2 \sum_{i<j}^{n_k} \Phi(\{ u_{ii}^k, y , u_{jj}^k \}) \right)$$
$$ = \sum_{k\in I} \left( \sum_{i=1}^{n_k} \Phi(u_{ii}^k) \Phi( x )  \Phi(u_{ii}^k) + 2 \sum_{i<j}^{n_k} \{ \Phi(u_{ii}^k), \Phi(x) , \Phi(u_{jj}^k) \}) \right. $$ $$ \left.+ \sum_{i=1}^{n_k} \Phi(u_{ii}^k) \Phi(y) \Phi(u_{ii}^k) + 2 \sum_{i<j}^{n_k} \{ \Phi(u_{ii}^k), \Phi(y) , \Phi(u_{jj}^k) \}) \right)$$
$$= \sum_{k\in I} \left( \sum_{i,j=1}^{n_k} \Phi(u_{ii}^k) \Phi (x  ) \Phi(u_{jj}^k) + \sum_{i,j=1}^{n_k} \Phi(u_{ii}^k) \Phi (y ) \Phi(u_{jj}^k) \right) $$
$$=   \sum_{k\in I} \Phi(p_k) \Phi (x  ) \Phi(p_k) + \sum_{k\in I} \Phi(p_k) \Phi ( y ) \Phi(p_k) = \Phi (x) + \Phi(y).$$
\end{proof}

We consider next imaginary scalars.

\begin{lemma}\label{l imaginary numbers} Let $\Phi:M\to N$ be a bijective map between von Neumann algebras satisfying hypothesis $(h.4)$ in Problem \ref{problems}. Then there exists a central projection $p_c$ in $M$ satisfying the following properties:\begin{enumerate}[$(a)$]\item $\Phi (i p_c) = i \Phi (p_c)$;
\item $\Phi (i x) = i \Phi (x)$ for all $x\in p_c M_{sa}$;
\item $\Phi (i (\textbf{1}-p_c)) = - i \Phi (\textbf{1}-p_c)$;
\item $\Phi (i x) = - i \Phi (x)$ for all $x\in (\textbf{1}-p_c) M_{sa}$.
\end{enumerate}
\end{lemma}

\begin{proof} We begin with the case in which $x=1$. By Proposition \ref{p consequences preservation of jordan products for positive in h.4}$(d)$ we know that $\Phi (i \textbf{1})\in i N_{sa}$. Applying 
Proposition \ref{p basic properties}$(b)$ we know that $\Phi (i \textbf{1}) \circ \Phi (i \textbf{1})^* = \Phi (\textbf{1}) = \textbf{1}$.
Standard arguments in spectral theory show that $\Phi (i \textbf{1}) = i q - i (\textbf{1}-q)$, where $q$ is a projection in $N$.

Let us take a projection $p_c$ in $M$ satisfying $\Phi (p_c) = q$ (see Proposition \ref{p basic properties}$(c)$). An application of Proposition \ref{p consequences preservation of jordan products for positive in h.4}$(c)$, $(e)$ and $(h)$ gives $$i q = \Phi (p_c) \circ \Phi(i \textbf{1}) = \Phi (i p_c),\ \hbox{ and } -i (\textbf{1}-q) = \Phi (\textbf{1}-p_c) \circ \Phi(i \textbf{1}) = \Phi (i (\textbf{1}-p_c)).$$

Therefore $$\Phi (i \textbf{1})= \Phi (i p_c + i (\textbf{1}-p_c)) =  \Phi (i p_c + i (\textbf{1}-p_c)) \circ \left( \Phi(p_c) + \Phi (\textbf{1}-p_c)\right) $$ $$= \Phi (i p_c + i (\textbf{1}-p_c)) \circ \Phi(p_c) + \Phi (i p_c + i (\textbf{1}-p_c)) \circ \Phi (\textbf{1}-p_c)  $$ $$= \Phi (i p_c ) + \Phi ( i (\textbf{1}-p_c)) = i q - i (\textbf{1}-q) = i (2 q - \textbf{1}).$$

Lemma \ref{lemma central elements} assures that $\Phi (i \textbf{1}) = i q - i (\textbf{1}-q)$ is a central element in $N$. Clearly $i \textbf{1} = i q + i (\textbf{1}-q)$ also is a central element in $N$. Therefore $q$ is a central projection in $N$. A new application of Lemma \ref{lemma central elements} to $\Phi$ proves that $p_c$ is a central projection in $M$. We have already proved statements $(a)$ and $(c)$.\smallskip

$(b)$ Let us take $x\in p_c M_{sa}$. By $(a),$ Corollary \ref{c hermitian and skew symmetric}, and Proposition \ref{p consequences preservation of jordan products for positive in h.4}$(e)$ we get $$\Phi (i x) = \Phi (i p_c \circ x) = \Phi (i p_c) \circ \Phi (x) = i \Phi ( p_c) \circ \Phi (x)= i \Phi ( p_c \circ x) =i \Phi(x).$$

Statement $(d)$ can be similarly obtained via $(b)$, Corollary \ref{c hermitian and skew symmetric}, and Proposition \ref{p consequences preservation of jordan products for positive in h.4}$(e)$.
\end{proof}

We are in position now to establish our main conclusion about maps satisfying hypothesis $(h.4)$ in Problem \ref{problems}.

\begin{theorem}\label{t linearity on hermitian for h4} Let $\Phi:M\to N$ be a bijective map between von Neumann algebras satisfying hypothesis $(h.4)$ in Problem \ref{problems}. Suppose $M$ is a von Neumann algebra which has no abelian direct summand. Then the restriction $\Phi|_{M_{sa}} : M_{sa}\to N_{sa}$ is a Jordan isomorphism.

If we also assume that $\Phi (x +i y ) = \Phi (x) +\Phi (i y)$ for all $x,y\in M_{sa}$, then there exists a central projection $p_c$ in $M$ such that $\Phi|_{p_c M}$ is a complex linear Jordan $^*$-isomorphism and $\Phi|_{(\textbf{1}-p_c) M}$ is a conjugate linear Jordan $^*$-isomorphism.
\end{theorem}

\begin{proof} Proposition \ref{p additivity on hermitian} and Corollary \ref{c hermitian and skew symmetric} prove that $\Phi|_{M_{sa}} : M_{sa}\to N_{sa}$ is an additive bijection and a Jordan multiplicative mapping. Lemma \ref{l additivity on hermitian gives lienarity on hermitian} assures that $\Phi|_{M_{sa}}$ is linear.\smallskip

Let $p_c$ be the  central projection in $M$ given by Lemma \ref{l imaginary numbers}. The just quoted result shows that $$\Phi (i x) = i \Phi (x) \hbox{ for all $x\in p_c M_{sa}$}$$ and
$$\Phi (i x) = - i \Phi (x) \hbox{ for all $x\in (\textbf{1}-p_c) M_{sa}$.}$$

Elements in $p_c M$ (respectively, $(\textbf{1}-p_c) M$) can be written in the form $z = z_1 + i z_2 $ and $y = y_1 + i y_2 $ with $z_1,z_2,y_1,y_2\in p_c M_{sa}$ (respectively, with $z_1,z_2,y_1,y_2\in (\textbf{1}-p_c) M_{sa}$). It follows from the hypothesis and the previous statements that $$\Phi (z + y ) = \Phi ( z_1 + i z_2 + y_1 + i y_2 ) = \Phi (z_1 + i z_2 + y_1 + i y_2 ) $$ $$= \Phi ( z_1 + y_1  ) + \Phi ( i z_2 + i y_2 ) = \Phi ( z_1) + \Phi ( y_1  ) + i \Phi (z_2) + i \Phi (y_2 )= \Phi(z) + \Phi (y)$$ (respectively, $$\Phi (z + y ) = \Phi ( z_1 + i z_2 + y_1 + i y_2 ) = \Phi (z_1 + i z_2 + y_1 + i y_2 ) $$ $$= \Phi ( z_1 + y_1  ) + \Phi ( i z_2 + i y_2 ) = \Phi ( z_1) + \Phi ( y_1  ) - i \Phi (z_2) - i \Phi (y_2 )= \Phi(z)^* + \Phi (y)^*).$$ Given $\alpha\in \mathbb{C}$, it follows from the above facts that $$\Phi (\alpha z) = \Phi (\alpha z_1 + \alpha i z_2) =   \Phi ((\Re\hbox{e} (\alpha) + i \Im\hbox{m}(\alpha)) z_1 + (i \Re\hbox{e} (\alpha) - \Im\hbox{m}(\alpha)) z_2)$$ $$=   \Phi ( \Re\hbox{e} (\alpha)  z_1 - \Im\hbox{m}(\alpha) z_2) +   \Phi ( i \Im\hbox{m}(\alpha) z_1 + i \Re\hbox{e} (\alpha) z_2)$$ $$=    \Re\hbox{e} (\alpha)  \Phi (z_1)  - \Im\hbox{m}(\alpha) \Phi ( z_2) +  i \Im\hbox{m}(\alpha) \Phi (z_1) + i \Re\hbox{e} (\alpha) \Phi (z_2) = \alpha \Phi (z),$$ for all $z\in p_c M$, and $$\Phi (\alpha z) = \Phi (\alpha z_1 + \alpha i z_2) =   \Phi ((\Re\hbox{e} (\alpha) + i \Im\hbox{m}(\alpha)) z_1 + (i \Re\hbox{e} (\alpha) - \Im\hbox{m}(\alpha)) z_2)$$ $$=   \Phi ( \Re\hbox{e} (\alpha)  z_1 - \Im\hbox{m}(\alpha) z_2) +   \Phi ( i \Im\hbox{m}(\alpha) z_1 + i \Re\hbox{e} (\alpha) z_2)$$ $$=    \Re\hbox{e} (\alpha)  \Phi (z_1)  - \Im\hbox{m}(\alpha) \Phi ( z_2) -  i \Im\hbox{m}(\alpha) \Phi (z_1) - i \Re\hbox{e} (\alpha) \Phi (z_2) = \overline{\alpha} \Phi (z),$$ for all $z\in (\textbf{1}-p_c) M$, which concludes the proof.
\end{proof}

\section{Maps commuting with the C$^*$-product up to a $\lambda$-Aluthge transform}\label{sec:4}

In this section we shall consider bijective maps satisfying hypothesis $(h.3)$ in Problem \ref{problems}. Contrary to the case in previous section, we shall restrict our study to bijections $\Phi: \mathcal{B}(H)\to \mathcal{B}(K)$ satisfying $(h.3)$, that is, bijections for which there exists $\lambda\in [0,1]$ such that $$\Phi(\Delta_{\lambda}(a b^*))=\Delta_{\lambda}(\Phi(a) \Phi(b)^*), \hbox{ for all } a,\;b\in \mathcal{B}(H),$$ where $H$ and $K$ are complex Hilbert spaces. As we commented at the introduction, the case $\lambda=0$ was solved by J. Hakeda in \cite{Hak86a} (see Theorem \ref{thm Hakeda 86}).\smallskip

We shall extend and adapt some of the arguments given by F. Chabbabi in \cite{Chabbabi2017product} and F. Chabbabi and M. Mbekhta in \cite{ChabbabiMbekhta2017Jordan}.
We begin by establishing an analogous of \cite[Proposition 2.2]{Chabbabi2017product} and \cite[Proposition 3.2]{ChabbabiMbekhta2017Jordan} in our setting.

\begin{proposition}\label{p def function h} Let $\Phi:\mathcal{B}(H)\to \mathcal{B}(K)$ be a bijective map satisfying hypothesis $(h.3)$ in Problem \ref{problems} for a fixed  $\lambda$ in $(0,1]$.
Then there exists a bijective mapping $h:\CC \to \CC$ satisfying:
\begin{enumerate}[$(a)$]
\item $h(0)=0$ and $h(1)=1;$
\item $\Phi(\alpha \textbf{1})=h(\alpha) \textbf{1},$ for all $\alpha \in \CC;$
\item $h(\alpha \overline{\beta})=h(\alpha) \overline{h(\beta)},$ for all $\alpha,\;\beta \in \CC.$ In particular, $h(\overline{\alpha})=\overline{h(\alpha)},$ for all $\alpha \in \CC;$
\item $h(-\alpha)=-h(\alpha),$  for all $\alpha \in \CC;$
\item $\Phi(\alpha a)=h(\alpha)\Phi(a),$ for every hermitian operator $a\in\mathcal{B}(H)$ and every $\alpha\in\CC.$
\end{enumerate}
\end{proposition}

\begin{proof}  Most of the arguments in \cite[Proposition 3.2]{ChabbabiMbekhta2017Jordan} remain valid for our purposes, we include an sketch of the ideas for completeness reasons. By Proposition \ref{p basic properties}$(a)$ and $(e)$, $\Phi(0)=0$ and $\phi(\textbf{1})=\textbf{1}$. We can therefore set $h(0)=0$ and $h(1)=1.$ We consider now a nonzero scalar $\alpha\in \mathbb{C}.$ Let us take a minimal projection $p$ in $\mathcal{B}(H)$. By hypothesis we have $$\Delta_{\lambda} (\Phi(\alpha p) \Phi (p))= \Delta_{\lambda} (\Phi(\alpha p) \Phi (p)^*) =\Phi ( \alpha p ).$$ Lemma \ref{l Chabb 2.3 von Neumann} implies that $\Phi(\alpha p) \Phi (p) =  \Phi (p) \Phi(\alpha p)=  \Phi (p) \Phi(\alpha p)  \Phi (p) =  \Phi(\alpha p),$ and since $  \Phi (p)$ is a minimal projection (compare Proposition \ref{p basic properties}$(g)$), we deduce the existence of a unique $h_p(\alpha)\in \mathbb{C}\backslash\{0\}$ such that $ \Phi(\alpha p) = h_p(\alpha ) \Phi(p)$. We also set $h_p(0)=0$ and $h_p(1)=1.$\smallskip

We have therefore proved that $\Phi (\mathbb{C} p ) \subseteq \mathbb{C} \Phi(p),$ $\Phi|_{\mathbb{C} p} : \mathbb{C} p \to  \mathbb{C} \Phi(p)$ is a bijection (just apply the above conclusion to $\Phi^{-1}$), and $\Phi (\alpha p ) = h_p(\alpha) \Phi (p)$ for all $\alpha \in \mathbb{C}$. In particular $h_{p} : \mathbb{C} \to  \mathbb{C}$ is a bijection for every minimal projection $p$ in $\mathcal{B}(H)$. Furthermore, by the hypothesis and the above properties $$h_p(\alpha) \overline{h_p(\beta)} \Phi(p)= \Delta_{\lambda} (\Phi(\alpha p ) \Phi(\beta p)^*)= \Phi( \Delta_{\lambda} ( (\alpha p) (\beta p)^*))$$ $$ = \Phi (\alpha \overline{\beta} p) = h_p(\alpha \overline{\beta}) \Phi (p).$$ We deduce that \begin{equation}\label{eq hp is multiplicative} h_p(\alpha) \overline{h_p(\beta)} =  h_p(\alpha \overline{\beta}),
 \end{equation} for all $\alpha, \beta\in \mathbb{C}$.\smallskip

Now for each $\alpha \in \mathbb{C}\backslash\{0\}$ and each minimal projection $p\in \mathcal{B}(H)$ we have $$\Delta_{\lambda}(\Phi(\alpha \textbf{1}) \Phi(p)) = \Phi (\Delta_{\lambda} (\alpha p)) = \Phi (\alpha p ) = h_p (\alpha) \Phi (p).$$ Having in mind Proposition \ref{p basic properties}$(g)$, we conclude that $$\Delta_{\lambda}(\Phi(\alpha \textbf{1})(\xi)\otimes \xi) =\Delta_{\lambda}(\Phi(\alpha \textbf{1}) q) = h_{\Phi^{-1}(q)} (\alpha) q,$$ for every minimal projection $q=\xi \otimes \xi$ in $\mathcal{B}(K)$ with $\xi\in K$ and $\|\xi\|=1$. The element $\Phi(\alpha \textbf{1})(\xi)$ must be nonzero because $q$ and $h_{\Phi^{-1}(q)} (\alpha)$ both are nonzero. Proposition 2.1 in \cite{Chabbabi2017product} asserts that $$\Delta_{\lambda}(\Phi(\alpha \textbf{1})(\xi)\otimes \xi) = \langle \Phi(\alpha \textbf{1})(\xi)| \xi \rangle \xi\otimes \xi =\langle \Phi(\alpha \textbf{1})(\xi)| \xi \rangle q,$$ where $\langle.|.\rangle$ denotes the inner product in $K$. We deduce that \begin{equation}\label{eq phi(alpha 1) xi and xi} \langle \Phi(\alpha \textbf{1})(\xi)| \xi \rangle = h_{\Phi^{-1}(q)} (\alpha),
\end{equation} for every $\alpha\in \mathbb{C}\backslash\{0\}$.\smallskip

On the other hand, $$h_p(\alpha) \xi\otimes\xi =\Delta_{\lambda}(\Phi(\alpha \textbf{1}) \Phi(p))=\Phi(\Delta_{\lambda}(\alpha p))=\Phi(\Delta_{\lambda}( p (\overline{\alpha} \textbf{1})^*)$$ $$=\Delta_{\lambda}(\Phi(p)\Phi(\overline{\alpha} \textbf{1})^*)=\Delta_{\lambda}((\xi\otimes \xi) \Phi(\overline{\alpha} \textbf{1})^*)=\Delta_{\lambda}((\xi\otimes \Phi(\overline{\alpha} \textbf{1})(\xi)) ) $$ $$= \langle \xi | \Phi(\overline{\alpha} \textbf{1})(\xi) \rangle \ \|\Phi(\overline{\alpha} \textbf{1})(\xi)\|^{-2} \Phi(\overline{\alpha} \textbf{1})(\xi)\otimes \Phi(\overline{\alpha} \textbf{1})(\xi).$$ This proves that $\Phi(\overline{\alpha} \textbf{1})(\xi)$ and $\xi$ are linearly dependent, and thus, by \eqref{eq phi(alpha 1) xi and xi}, we get $\Phi(\overline{\alpha} \textbf{1})(\xi) = h_{\Phi^{-1}(\xi\otimes\xi)} (\overline{\alpha}) \xi$ and $\Phi({\alpha} \textbf{1})(\xi) = h_{\Phi^{-1}(\xi\otimes\xi)} ({\alpha}) \xi$, for every $\alpha\in \mathbb{C}$. We have established that $$\Phi (\alpha \textbf{1}) = h_p (\alpha) \textbf{1},$$ for every $\alpha\in \mathbb{C}$ and for every minimal projection $p$ in $\mathcal{B}(H)$.\smallskip

If we set $h: \mathbb{C} \to  \mathbb{C}$, $h(\alpha) = h_p(\alpha)$, where $p$ is any minimal projection in $\mathcal{B}(H)$, then the mapping $h$ is well defined and is a bijection satisfying statements $(a)$ and $(b)$ above. Statement $(c)$ is a consequence of \eqref{eq hp is multiplicative}.\smallskip

$(d)$ Since $h(-1) \textbf{1} = \Phi (-\textbf{1}) \in \mathcal{B}(K)_{sa}$ (see Corollary \ref{c hermitian and skew symmetric for h3 and h4}), and by hypothesis $$h(-1)^2 \textbf{1} =  \Delta_{\lambda} (\Phi(-\textbf{1}) \Phi(-\textbf{1})^*) = \Phi ((\textbf{-1})^2) = \Phi (\textbf{1}) = \textbf{1},$$ we deduce that $h(-1)^2=1,$ this implies that $h(-1)=1$ or $h(-1)=-1.$ However $h$ bijective and $h(1)=1$ imply $h(-1)=-1.$\smallskip

Now, let $\alpha \in\CC.$ By the above properties we get $h(-\alpha) =h(\alpha) h(-1)=-h(\alpha).$\smallskip

$(e)$ Let us take a hermitian operator $a\in \mathcal{B}(H)$. The previous statements and Corollary \ref{c hermitian and skew symmetric for h3 and h4} guarantee that $$h(\alpha) \Phi (a) = \Delta_{\lambda} (h(\alpha) \Phi (a)) = \Delta_{\lambda} (\Phi(\alpha \textbf{1}) \Phi (a)^*) = \Phi (\alpha a). $$\end{proof}

We insert next a technical exercise of linear algebra which will be required for latter purposes. Henceforth the symbol $M_2(\mathbb{C})$ will stand for the C$^*$-algebra of $2 \times 2$ matrices with complex entries.

\begin{lemma}\label{l tecnical lemma linear algebra in M2(C)} Let $a$ be an element in $M_2(\mathbb{C})$, and let $\widehat{p}$ be a minimal projection in $M_2(\mathbb{C})$. Suppose that $\Delta_{\lambda}( a (\textbf{1}-\widehat{p})) = \mu (\textbf{1}-\widehat{p})$ and $\Delta_{\lambda}( (\textbf{1}-\widehat{p}) a^*) = \overline{\mu} (\textbf{1}-\widehat{p}),$ where $\mu$ is a nonzero complex number. Then $ \widehat{p}   a (\textbf{1}-\widehat{p})=0.$
\end{lemma}

\begin{proof} We may assume, without any loss of generality, that $\widehat{p} = \left(
                                                                                   \begin{array}{cc}
                                                                                     1 & 0 \\
                                                                                     0 & 0 \\
                                                                                   \end{array}
                                                                                 \right),$ $\textbf{1}-\widehat{p} = \left(
                                                                                                                      \begin{array}{cc}
                                                                                                                        0 & 0 \\
                                                                                                                        0 & 1 \\
                                                                                                                      \end{array}
                                                                                                                    \right)
,$ and $a= \left(
             \begin{array}{cc}
               \alpha_{11} & \alpha_{12} \\
               \alpha_{21} & \alpha_{22} \\
             \end{array}
           \right).$ It will be enough to show that $\alpha_{12} =0$. Arguing by contradiction we assume $\alpha_{12}\neq 0$.\smallskip

In our setting we have $$a (\textbf{1}-\widehat{p}) = \left(
             \begin{array}{cc}
               0 & \alpha_{12} \\
               0 & \alpha_{22} \\
             \end{array}\right),\ |a (\textbf{1}-\widehat{p})|^2 = \lambda_0^2 (\textbf{1}-\widehat{p}),$$  with $\lambda_0=\sqrt{|\alpha_{12}|^2 + |\alpha_{22}|^2}>0$, the polar decomposition of $ a (\textbf{1}-\widehat{p})$ is $$a (\textbf{1}-\widehat{p}) = u_1 |a (\textbf{1}-\widehat{p})|, \hbox{ with } u_1= \lambda_0^{-1} \left(
             \begin{array}{cc}
             0 & \alpha_{12} \\
             0 & \alpha_{22} \\
             \end{array}
             \right),
             $$
 and hence the condition $\Delta_{\lambda}( a (\textbf{1}-\widehat{p})) = \mu (\textbf{1}-\widehat{p})$ is equivalent to
 $$\alpha_{22} (\textbf{1}-\widehat{p})= \lambda_0^{\lambda}\ (\textbf{1}-\widehat{p})\ u_1\ \lambda_0^{1-\lambda}\ (\textbf{1}-\widehat{p}) = \Delta_{\lambda}( a (\textbf{1}-\widehat{p})) = \mu (\textbf{1}-\widehat{p}),$$ which implies $\mu = \alpha_{22}\neq 0$.\smallskip

On the other hand
$$(\textbf{1}-\widehat{p}) a^*=\left(
             \begin{array}{cc}
               0 & 0 \\
               \overline{\alpha_{12}} & \overline{\alpha_{22}} \\
             \end{array}
           \right),\  |(\textbf{1}-\widehat{p}) a^*| = \lambda_0 p_0,$$ where $p_0$ is the projection given by $p_0 = \lambda_0^{-2} \left(
             \begin{array}{cc}
               |\alpha_{12}|^2 & {\alpha_{12}} \overline{\alpha_{22}} \\
               \overline{\alpha_{12}} {\alpha_{22}} & |\alpha_{22}|^2 \\
             \end{array}
           \right) $, and the polar decomposition of $(\textbf{1}-\widehat{p}) a^*$ is given by $(\textbf{1}-\widehat{p}) a^* = u_2 |(\textbf{1}-\widehat{p}) a^*|$ with $u_2 = \lambda_0^{-1} \left(
                                                                        \begin{array}{cc}
                                                                          0 & 0 \\
                                                                          \overline{\alpha_{12}} & \overline{\alpha_{22}} \\
                                                                        \end{array}
                                                                      \right)
           $. According to this, the equation $\Delta_{\lambda} ( (\textbf{1}-\widehat{p}) a^*) = \overline{\mu} (\textbf{1}-\widehat{p})$ rewrites in the form
$$\lambda_0^{-4} \left(
             \begin{array}{cc}
               |\alpha_{12}|^2 & {\alpha_{12}} \overline{\alpha_{22}} \\
               \overline{\alpha_{12}} {\alpha_{22}} & |\alpha_{22}|^2 \\
             \end{array}
           \right) \left(
           \begin{array}{cc}
            0 & 0 \\
            \overline{\alpha_{12}} (|\alpha_{12}|^2+|\alpha_{22}|^2) & \overline{\alpha_{22}}  (|\alpha_{12}|^2+|\alpha_{22}|^2) \\
             \end{array}
              \right)$$ $$= \lambda_0^{-4} \left(
             \begin{array}{cc}
               |\alpha_{12}|^2 & {\alpha_{12}} \overline{\alpha_{22}} \\
               \overline{\alpha_{12}} {\alpha_{22}} & |\alpha_{22}|^2 \\
             \end{array}
           \right) \left(
                                                                        \begin{array}{cc}
                                                                          0 & 0 \\
                                                                          \overline{\alpha_{12}} & \overline{\alpha_{22}} \\
                                                                        \end{array}
                                                                      \right) \left(
             \begin{array}{cc}
               |\alpha_{12}|^2 & {\alpha_{12}} \overline{\alpha_{22}} \\
               \overline{\alpha_{12}} {\alpha_{22}} & |\alpha_{22}|^2 \\
             \end{array}
           \right)    $$ $$=\lambda_0\ p_0 u_2 p_0= \lambda_0^{\lambda}\ p_0\  u_2\  \lambda_0^{1-\lambda}\  p_0= \Delta_{\lambda} ( (\textbf{1}-\widehat{p}) a^*)= \overline{\mu} (\textbf{1}-\widehat{p})=  \overline{\mu} \left(
                                                                                                     \begin{array}{cc}
                                                                                                       0 & 0 \\
                                                                                                       0 & 1 \\
                                                                                                     \end{array}
                                                                                                   \right),
           $$ the equality among the elements in the $(2,1)$ entries gives $ \lambda_0^{-2}\ |\alpha_{22}|^2 \overline{\alpha_{12}} =0,$ which is impossible.
\end{proof}

Our next result has been more or less explicitly established in \cite[Lemma 2.5]{Chabbabi2017product} for an arbitrary bijection $\Phi:\mathcal{B}(H)\to \mathcal{B}(K)$ satisfying \eqref{eq product commuting maps}, that is, $\Phi(\Delta_{\lambda}(a b))=\Delta_{\lambda}(\Phi(a) \Phi(b)),$ for all $a, b\in \mathcal{B}(H),$ and a fixed $\lambda\in (0,1)$. Instead of the commented hypothesis, we consider a bijection $\Phi:\mathcal{B}(H)\to \mathcal{B}(K)$ satisfying hypothesis $(h.3)$ in Problem \ref{problems} for a fixed  $\lambda$ in $(0,1]$, and we offer an independent proof which is also valid for the just quoted result.

\begin{lemma}\label{lem1.5}
Let $\Phi:\mathcal{B}(H)\to \mathcal{B}(K)$ a bijection satisfying hypothesis $(h.3)$ in Problem \ref{problems} for a fixed  $\lambda$ in $(0,1]$. Suppose $p$ and $q$ are minimal projections in $\mathcal{B}(H)$ with $p \perp q.$ Then the identity
$$\Phi(\alpha p+\beta q)=\Phi(\alpha p) + \Phi(\beta q)= h(\alpha) \Phi(p)+ h(\beta) \Phi(q),$$ holds for all $\alpha,\beta$ in $\mathbb{C}$, where $h$ is the function defined by Proposition \ref{p def function h}.
\end{lemma}

\begin{proof} Let us fix $\alpha,\beta$ in $\mathbb{C}\backslash \{0\},$ and set $a = \Phi(\alpha p+\beta q)$. By Proposition \ref{p basic properties}$(f),$ $(g)$ and $(h)$, we know that $\widehat{p}=\Phi(p)$ and $\widehat{q}=\Phi(q)$ are mutually orthogonal minimal projections in $\mathcal{B}(K)$ with $\widehat{p}+\widehat{q} =\Phi(p) +\Phi (q) = \Phi (p+q)$. By hypothesis $$\Delta_{\lambda} (a (\widehat{p}+\widehat{q})) = \Delta_{\lambda} (\Phi(\alpha p+\beta q) \Phi (p+q)^*) = \Phi( \Delta_{\lambda} ( (\alpha p+\beta q) (p+q)^*)) $$ $$=   \Phi(\alpha p+\beta q) = a.$$ Lemma \ref{l Chabb 2.3 von Neumann} assures that $ a  (\widehat{p}+\widehat{q})=  (\widehat{p}+\widehat{q})a =  (\widehat{p}+\widehat{q})a  (\widehat{p}+\widehat{q})=a$. We can therefore regard $a$, $\widehat{p}$, $\widehat{q}$ and $\widehat{p}+\widehat{q}$ inside $ (\widehat{p}+\widehat{q}) \mathcal{B}(K) (\widehat{p}+\widehat{q}) = M_2 (\mathbb{C})$.\smallskip

We also know from the hypothesis and Proposition \ref{p def function h} that \begin{equation}\label{eq 1 1212} \Delta_{\lambda} (a \widehat{p}) = \Delta_{\lambda} (\Phi(\alpha p+\beta q) \Phi (p)^*) = \Phi( \Delta_{\lambda} ( (\alpha p+\beta q) p^*))
\end{equation} $$=   \Phi(\alpha p) = h(\alpha) \Phi (p) = h(\alpha) \widehat{p},$$
\begin{equation}\label{eq 2 1212} \Delta_{\lambda} (\widehat{p} a^*) = \Delta_{\lambda} (\Phi (p) \Phi(\alpha p+\beta q)^*) = \Phi( \Delta_{\lambda} ( p (\alpha p+\beta q)^*))
\end{equation} $$=   \Phi(\overline{\alpha} p) = h(\overline{\alpha}) \Phi (p) = \overline{h(\alpha)} \widehat{p},$$ and similarly, \begin{equation}\label{eq 3 1212} \Delta_{\lambda} (a \widehat{q})= \Phi(\beta q) = h(\beta) \widehat{q},\hbox{ and } \Delta_{\lambda} (\widehat{q} a^*)= \Phi(\overline{\beta} q) = \overline{h(\beta)} \widehat{q}.
 \end{equation} By applying Lemma \ref{l tecnical lemma linear algebra in M2(C)} to the pairs $(a,\widehat{p})$ and $(a, \widehat{q})$ we deduce, via \eqref{eq 1 1212}, \eqref{eq 2 1212} and \eqref{eq 3 1212}, that $\widehat{p} a \widehat{q} = \widehat{q} a \widehat{p}=0,$ and consequently $$ \Phi(\alpha p+\beta q) =a = \widehat{p} a \widehat{p} + \widehat{q} a \widehat{q} = a \widehat{p} +  a \widehat{q}.$$

Finally, by the minimality of $\widehat{p}$ and $\widehat{q},$ we know that $a  \widehat{p} = \widehat{p} a \widehat{p} = \alpha_{a} \widehat{p},$ 
and $a  \widehat{q} = \widehat{q} a \widehat{q} = \beta_{a} \widehat{q},$ 
for unique $\alpha_a,\beta_a$ in $\mathbb{C}$. Combining this information with \eqref{eq 1 1212} and \eqref{eq 3 1212} we get $ a \widehat{p}= \Delta_{\lambda} (a \widehat{p}) = \Phi(\alpha p)=   h(\alpha) \Phi (p) =  h(\alpha) \widehat{p}$ and $ a \widehat{q}= \Delta_{\lambda} (a \widehat{q}) = \Phi({\beta} q) = h(\beta) \Phi (q) = h(\beta) \widehat{q},$ and thus $$ \Phi(\alpha p+\beta q) =a =  a \widehat{p} +  a \widehat{q} =\Phi(\alpha p) + \Phi(\beta q) = h(\alpha) \Phi (p)+ h(\beta) \Phi (q).$$
\end{proof}

We require at this stage the existence of at least two mutually orthogonal minimal projections in $\mathcal{B}(H)$ to guarantee that the mapping $h$ given by Proposition \ref{p def function h} is additive. For each minimal projection $p= \xi\otimes \xi\in \mathcal{B}(H)$ (where $\xi$ is a norm-one element in $H$), the symbol $\varphi_p$ will denote the trace class functional defined by $\varphi_p (a) = \langle a(\xi)|\xi \rangle$ ($a\in \mathcal{B}(H)$). The functional $\varphi_p$ is the unique positive normal state in $\mathcal{B}(H)_{*}$ satisfying $\varphi_p (p) =1.$ Pure states in the predual of a von Neumann algebra $M$ are in one-to-one correspondence with minimal projections in $M$ (see \cite[Proposition 3.13.6]{Ped}).

\begin{lemma}\label{l additivity of h} Let $H$ and $K$ be complex Hilbert spaces with dim$(H)\geq 2$. Suppose $\Phi:\mathcal{B}(H)\to \mathcal{B}(K)$ a bijection satisfying hypothesis $(h.3)$ in Problem \ref{problems} for a fixed  $\lambda$ in $(0,1]$. Let $h$ be the mapping defined by Proposition \ref{p def function h}. Then the following statements hold:\begin{enumerate}[$(a)$]\item For each minimal projection $p$ in $\mathcal{B}(H)$ the identity
$$\varphi_{\Phi(p)} (\Phi(a)) \Phi (p)=\Phi (p) \Phi(a) \Phi(p)  = h (\varphi_p (a)) \Phi (p),$$ holds for all $a\in\mathcal{B}(H);$
\item The function $h$ is additive;
\item $h$ is the identity or the conjugation on $\mathbb{C}$.
\end{enumerate}
\end{lemma}

\begin{proof}$(a)$ Every minimal projection $p$ in $\mathcal{B}(H)$ writes in the form $p= \xi\otimes \xi$ with $\xi$ in the unit sphere of $H$. We know from Proposition \ref{p basic properties}$(g)$ that $\Phi (p) = \eta\otimes \eta$ for a unique $\eta$ in the unit sphere of $K$. The pure states $\varphi_p$ and $\varphi_{\Phi(p)}$ are completely described by $\xi$ and $\eta$, respectively.\smallskip

By hypothesis, for each $a\in\mathcal{B}(H),$ we have
$$\Delta_{\lambda}(\Phi(a) (\eta) \otimes \eta) = \Delta_{\lambda}(\Phi(a) (\eta\otimes \eta) )=\Delta_{\lambda}(\Phi(a) \Phi(p))$$ $$=\Phi(\Delta_{\lambda}(a (\xi\otimes \xi)))=\Phi(\Delta_{\lambda}(a(\xi)\otimes \xi )).$$
Now, by \cite[Propostion 2.1]{Chabbabi2017product} we have
$$\Delta_{\lambda}(\Phi(a) (\eta) \otimes \eta) =\langle\Phi(a)(\eta) | \eta \rangle (\eta\otimes \eta),\hbox{ and } \Delta_{\lambda}(a(\xi)\otimes \xi )=\langle a(\xi)|\xi\rangle (\xi\otimes \xi)$$ and thus $$\varphi_{\Phi(p)} (\Phi(a)) \Phi (p)=\Phi (p) \Phi(a) \Phi(p)  = \langle\Phi(a)(\eta) | \eta \rangle (\eta\otimes \eta)= \Phi (\langle a(\xi)|\xi\rangle (\xi\otimes \xi) )  $$ $$= h(\langle a(\xi)|\xi\rangle) \Phi(\xi\otimes \xi) = h(\langle a(\xi)|\xi\rangle) (\eta\otimes \eta) = h (\varphi_p (a)) \Phi (p).$$

$(b)$ Let us pick two orthogonal minimal projections $p,q\in\mathcal{B}(H)$ (here we require that dim$(H)\geq 2$). We can find two minimal partial isometries $e_{12} = e_{21}^*$ satisfying $e_{12} e_{12}^* = p$ and $e_{12}^* e_{12}= q.$ Let us consider the minimal projections $v_1 = \frac12 (p + e_{12} + e_{21}+ q)$ and $v_2 = \frac12 (p - e_{12} - e_{21}+ q)$. Clearly, $v_1 \perp v_2$. By $(a)$ and Proposition \ref{p def function h} it follows that $$ \Phi (p) \Phi(\alpha v_1 + \beta v_2) \Phi(p) = h (\varphi_p (\alpha v_1 + \beta v_2)) \Phi (p) = h\left(\frac{\alpha+\beta}{2}\right)  \Phi (p).$$ On the other hand, by Lemma \ref{lem1.5} $\Phi(\alpha v_1 + \beta v_2) = \Phi(\alpha v_1) + \Phi( \beta v_2)$, and consequently, by $(a)$ and the proerties of $h$, we derive at
$$h\left(\frac{\alpha+\beta}{2}\right)  \Phi (p)= \Phi (p) \Phi(\alpha v_1 + \beta v_2) \Phi(p) $$ $$= \Phi (p) \Phi(\alpha v_1 ) \Phi(p) + \Phi (p) \Phi(\beta v_2) \Phi(p) = $$
$$=h (\varphi_p (\alpha v_1 )) \Phi (p) + h (\varphi_p ( \beta v_2)) \Phi (p) = h\left(\frac{\alpha}{2}\right)  \Phi (p) + h\left(\frac{\beta}{2}\right)  \Phi (p),$$ which assures that $$h(\alpha+\beta) h\left(\frac12\right) = h\left(\frac{\alpha+\beta}{2}\right) =  h\left(\frac{\alpha}{2}\right) + h\left(\frac{\beta}{2}\right)= ( h(\alpha)+h(\beta) )  h\left(\frac12\right).$$

$(c)$ We know from the above and Proposition \ref{p def function h} that $h:\mathbb{C}\to \mathbb{C}$ is an additive bijection satisfying \begin{enumerate}[$(1)$] \item $h(0)=0$, $h(1) =1$;
\item $h(\alpha \overline{\beta}) = h(\alpha) \overline{h(\beta)}$ for all $\alpha, \beta$;
\item $h(\overline{\alpha}) = \overline{h(\alpha)}$ for all $\alpha$.
\end{enumerate} We observe that $h(\alpha {\beta}) = h(\alpha \overline{\overline{\beta}}) = h(\alpha) \overline{h(\overline{\beta})} = h(\alpha) {h(\beta)}$ for all $\alpha, \beta,$ that is, $h$ is multiplicative.\smallskip

We shall next prove that $h(-z) =- h(z),$ for all $z\in \mathbb{C}$. Indeed, by the additivity of $h$ we get $0 = h(0) = h(z-z) = h (z) + h(-z),$ and hence $h(-z) = -h(z)$.\smallskip

Now, we shall show that $h(\mathbb{R})\subseteq \mathbb{R}$ and $h(i \mathbb{R})\subseteq i \mathbb{R}$. Namely, if $x\in \mathbb{R}$ then, by $(4)$, we have $$h(x ) = h(\overline{x}) = \overline{h(x)},$$ $$ h(i x ) = h(- \overline{i x})=  - h( \overline{i x}) = - \overline{h(i x)},$$ which proves the desired statement.\smallskip

Lemma \ref{l additivity on hermitian gives lienarity on hermitian} assures that $h|_{\mathbb{R}} : \mathbb{R}\to \mathbb{R}$ is a linear mapping, and thus, $h(x) = x,$ for all $x\in \mathbb{R}$. \smallskip

Finally, we conclude from $(2)$ that $|h(i)|^2 = h(i) \overline{h(i)} = h(i) h(\overline{i}) = h( i \overline{i}) = h(1) =1$. Therefore, since $i \mathbb{R}\ni h(i)$, it follows that $h(i) \in \{i,-i\}$. If $h(i) = i$ (respectively, $h(i)= -i$) we get from the above properties that
$$ h(x + i y) = h(x) + h(i y) = x + h(i) h(y) = x + i y,$$ (respectively, $ h(x + i y) = h(x) + h(i y) = x + h(i) h(y) = x - i y)$ for all $x,y\in \mathbb{R}$, which finishes the proof.
\end{proof}

We can now estate our main conclusion about bijections from $\mathcal{B}(H)$ onto $\mathcal{B}(K)$ satisfying hypothesis $(h.3)$ in Problem \ref{problems} for a fixed  $\lambda$ in $(0,1]$.

\begin{theorem}\label{thm h3 in problems} Let $H$ and $K$ be complex Hilbert spaces with dim$(H)\geq 2$. Suppose $\Phi:\mathcal{B}(H)\to \mathcal{B}(K)$ a bijection. Then the following statements are equivalent:
\begin{enumerate}[$(a)$]\item Satisfies hypothesis $(h.3)$ in Problem \ref{problems} for a fixed  $\lambda$ in $[0,1]$;
\item $\Phi$ is a complex linear or conjugate linear $^*$-isomorphism, that is, $\Phi$ is a complex linear or conjugate linear bijection, $\Phi(a b) = \Phi (a) \Phi(b)$, and $\Phi (a^*) = \Phi (a)^*$, for all $a,b\in \mathcal{B}(H)$.
\end{enumerate}
 \end{theorem}

\begin{proof} The implication $(b)\Rightarrow(a)$ is an immediate consequence of what we commented in \eqref{eq linear or conjugate linear *isomorphisms preserves lambda aluthge products}.\smallskip

$(a)\Rightarrow(b)$ The case $\lambda=0$ is solved by J. Hakeda \cite{Hak86a} (see Corollary \ref{c Hakeda Saito 86 for (4) or (3)}). We shall assume that $\lambda\in (0,1]$. By Lemma \ref{l additivity of h}, the mapping $h:\mathbb{C}\to \mathbb{C}$ defined by Proposition \ref{p def function h} is the identity or the conjugation on $\mathbb{C}$.\smallskip

We shall next show that $\Phi$ is linear or anti-linear. Let us pick a couple of norm-one elements $\xi\in H$ and $\eta\in K$ such that $\Phi(\xi\otimes \xi)=\eta\otimes \eta.$ Given $a,$ $b\in\mathcal{B}(H),$ Lemma \ref{l additivity of h}$(a)$ we have
$$\varphi_{\Phi(p)} (\Phi(a+b)) \Phi(p)=\Phi (p) \Phi(a+b) \Phi(p)  = h (\varphi_p (a+b)) \Phi (p)$$ $$=h (\varphi_p (a)) \Phi (p) + h (\varphi_p (b)) \Phi (p)=\varphi_{\Phi(p)} (\Phi(a)) \Phi(p) + \varphi_{\Phi(p)} (\Phi(b)) \Phi(p), $$ which implies that $\varphi_{\Phi(p)} (\Phi(a+b)) =\varphi_{\Phi(p)} (\Phi(a) +\Phi(b)),$ for every minimal projection $p$ in $\mathcal{B}(H)$. Since $\Phi$ maps the set of all minimal projections in $\mathcal{B}(H)$ to the set of all minimal projections in $\mathcal{B}(K),$ and pure states on $\mathcal{B}(K)$ separate points in the latter von Neumann algebra (see, for example, \cite[Lemma I.9.10 and Corollary I.9.11]{Davidson96}), we conclude that $\Phi(a+b) =\Phi(a) +\Phi(b)$, that is, $\Phi$ is additive.\smallskip

Since, by Proposition \ref{p def function h}$(d)$, $\Phi(\alpha a)=h(\alpha)\Phi(a),$ for every hermitian operator $a\in\mathcal{B}(H)$ and every $\alpha\in\CC,$ we can easily deduce that $\Phi$ is complex linear if $h$ is the identity on $\mathbb{C},$ or conjugate linear if $h$ is the conjugation on $\mathbb{C}$. Furthermore, it follows from Corollary \ref{c hermitian and skew symmetric for h3 and h4} that $\Phi$ is a symmetric mapping, that is, $\Phi(a^*) = \Phi (a)^*$ for all $a\in \mathcal{B}(H)$. Furthermore, by Lemma \ref{l additivity of h}$(a)$ and \cite[Lemma I.9.10]{Davidson96}, we have $\|\Phi (a) \|\leq 2 \|a\|$, for all $a\in \mathcal{B}(H)$, which assures that $\Phi$ is continuous.\smallskip

If $\Phi$ is conjugate linear, we can always find a conjugate linear $^*$-automor-phism $\overline{\cdot}: \mathcal{B}(K)\to \mathcal{B}(K)$ (we can actually find a conjugate linear bijection satisfying $\overline{a b} = \overline{a} \overline{b}$, $\overline{\overline{a}}=a$ and $(\overline{a})^* = \overline{a^*}$ for all $a,b\in \mathcal{B}(K)$). We consider in this case the mapping $\overline{\Phi}:  \mathcal{B}(H)\to \mathcal{B}(K),$ $\overline{\Phi} (a) = \overline{\Phi(a)}$, which is complex linear and satisfies the hypothesis $(h.3)$ in Problem \ref{problems} for a fixed  $\lambda$ in $[0,1]$. If we show that $\overline{\Phi}$ is a linear $^*$-isomorphism then $\Phi$ will be a conjugate linear $^*$-isomorphism.\smallskip

We are in an optimal position to apply the results obtained by F. Botelho, L. Moln{\'a}r and G. Nagy  \cite{BoteMolNag2016}. If $\Phi$ (respectively, $\overline{\Phi}$) is complex linear, Proposition \ref{p basic properties}$(j)$ guarantees that $\Phi$ preserves $\lambda$-Aluthge transforms, and thus, by applying \cite[Theorem 1]{BoteMolNag2016} (see also the introduction), we deduce that one of the next statements holds:
\begin{enumerate}[$(a)$]\item If dim$(H)>2$ there exists a complex linear $^*$-isomorphism $\Theta : \mathcal{B}(H)\to \mathcal{B}(K)$ and a nonzero scalar $c\in \mathbb{C}$ such that $\Phi (a) = c \Theta (a)$ (respectively, $\overline{\Phi} (a) = c \Theta (a)$), for all $a\in \mathcal{B}(H)$. Since $\Phi(\textbf{1})=\textbf{1}$ (respectively, $\overline{\Phi}(\textbf{1})=\textbf{1}$), it follows that $c=1$ and $\Phi = \Theta$ (respectively, $\overline{\Phi}= \Theta$) is a complex linear $^*$-isomorphism.
\item If dim$(H)=2$, then dim$(K)=2$ and there exists a complex linear  $^*$-isomorphism $\Theta : \mathcal{B}(H)\to \mathcal{B}(K)$ and a nonzero scalar $c\in \mathbb{C}$ such that $\Phi$ (respectively, $\overline{\Phi}$) is either of the form $$ \Phi(a) = c \ \Theta(a) \hbox{ (respectively, $\overline{\Phi} (a) = c \ \Theta(a)$)}, \hbox{ for all } a\in \mathcal{B}(H),$$ or of the form
$$\Phi ( a ) = c \left(\Theta(a^t) - \hbox{Tr}(a) \textbf{1}\right), \hbox{ for all } a\in \mathcal{B}(H),$$  (respectively, $\overline{\Phi} (a) = c \left(\Theta(a^t) - \hbox{Tr}(a) \textbf{1}\right)$ for all $a\in \mathcal{B} (H)$), where $^t$ stands for the transpose and Tr stands for the normalized trace functional on matrices.\smallskip

In the first case, the condition $\Phi(\textbf{1})=\textbf{1}$ (respectively, $\overline{\Phi}(\textbf{1})=\textbf{1}$), implies that $\Phi = \Theta$ (respectively, $\overline{\Phi}= \Theta$) is a complex linear $^*$-isomorphism.\smallskip

In the second case we have $\textbf{1}=\Phi(\textbf{1}) =  c \left(\Theta(\textbf{1}^t) - \hbox{Tr}(\textbf{1}) \textbf{1}\right) =\textbf{1},$ which is impossible.
\end{enumerate}
\end{proof}

Let $\Phi: M_n(\mathbb{C})\to M_n(\mathbb{C})$ be the mapping defined by $\Phi(a) =\Phi ((a_{ij})_{ij}) =  (\overline{a_{ij}})_{ij} = \overline{a}$. It is easy to see that $\Phi$ is a conjugate linear $^*$-automorphism on $M_n(\mathbb{C})$ and satisfies all hypothesis $(h.1)$-to-$(h.3)$ in Problem \ref{problems}, but we cannot conclude that $\Phi$ is complex linear.\smallskip

We shall conclude this note by presenting an additional connection with another result due to L. Moln{\'a}r in \cite{Mol2002}. We recall that elements $a,b$ in a C$^*$-algebra $A$ are called \emph{orthogonal} (written $a\perp b$) if $a b^* = b^* a=0$. This definition is consistent with the notion applied before for projections. The set of all partial isometries in $A$ (denoted by $\mathcal{P}I(A)$) can be equipped with a partial order ``$\leq$'' defined by $e\leq v$ if $v-e$ is a partial isometry orthogonal to $e$ (equivalently, $ee^* \leq vv^*$ and $e^*e\leq v^*v$). We shall write $\mathcal{P}I(H)$ for the set $\mathcal{P}I(\mathcal{B}(H))$. Theorem 1 in \cite{Mol2002} proves that, for any complex Hilbert space $H$ with dim$(H)\geq 3$, and every bijective transformation $\Psi : \mathcal{P}I(H)\to \mathcal{P}I(H)$ which preserves the partial ordering and the orthogonality between partial isometries in both directions, and is continuous (in the operator norm) at a single element of $\mathcal{P}I(H)$ different from $0$, then $\Psi$ can be written in one of the following forms:\begin{enumerate}[$(a)$] \item $\Psi (e) = T_1 (e)$ for all $e\in \mathcal{P}I(H)$, where $T_1$ is a $^*$-automorphism on $\mathcal{B}(H)$;
\item $\Psi (e) = T_2 (e)$ for all $e\in \mathcal{P}I(H)$, where $T_2$ is a $^*$-anti-automorphism on $\mathcal{B}(H)$;
\item $\Psi (e) = T_1 (e^*)$ for all $e\in \mathcal{P}I(H)$, where $T_1$ is a $^*$-automorphism on $\mathcal{B}(H)$;
\item $\Psi (e) = T_2 (e^*)$ for all $e\in \mathcal{P}I(H)$, where $T_2$ is a $^*$-anti-automorphism on $\mathcal{B}(H)$.
\end{enumerate}

For bijective mappings on $\mathcal{B}(H)$ satisfying  hypothesis $(h.3)$ in Problem \ref{problems} the above result of Moln{\'a}r can be applied to replace \cite[Theorem 1]{BoteMolNag2016} in the final part of the argument exhibited in the proof of Theorem \ref{thm h3 in problems}. Namely,
suppose $\Phi:\mathcal{B}(H)\to \mathcal{B}(H)$ is a bijection satisfying  hypothesis $(h.3)$ in Problem \ref{problems} for a fixed  $\lambda$ in $[0,1]$. By applying the arguments in the first part of the proof of Theorem \ref{thm h3 in problems} $(a)\Rightarrow (b)$, we can conclude that $\Phi$ is a continuous complex linear or a conjugate linear symmetric map (i.e. $\Phi (a^*) = \Phi (a^*)$ for all $a\in \mathcal{B}(H)$).\smallskip

Let us take a partial isometry $e\in \mathcal{P}I(H)$, by Proposition \ref{p basic properties}$(b)$ and $(f)$, we deduce that $$\Phi (e) \Phi(e)^*= \Phi(e e^*)\in \hbox{Proj} (\mathcal{B} (H)),$$ and hence $\Phi (e)$ is a partial isometry. We have therefore proved that \begin{equation}\label{eq Phi preserves partial isometries} \Phi \left(\mathcal{P}I(H) \right) = \mathcal{P}I(H),\hbox{  and } \Phi|_{\mathcal{P}I(H)}: \mathcal{P}I(H) \to \mathcal{P}I(H) \hbox{ is a bijection}.
\end{equation}

Suppose $e$ and $v$ are orthogonal partial isometries in $\mathcal{B} (H)$. Since $ee^*$ and $vv^*$ are orthogonal projections, Proposition \ref{p basic properties}$(f)$ assures that $\Phi (e) \Phi(e)^*= \Phi(e e^*)\perp  \Phi(v v^*) = \Phi (v) \Phi(v)^*)$. On the other hand, since  $e^*e$ and $v^*v$ are orthogonal projections too, by Corollary \ref{c hermitian and skew symmetric for h3 and h4} and Proposition \ref{p basic properties}$(b)$ and $(f),$ we know that
$$\Phi (e)^* \Phi(e)= \Phi (e^*) \Phi(e^*)^*= \Phi(e^* e)\perp  \Phi(v^* v) = \Phi (v^*) \Phi(v^*)^* = \Phi (v)^* \Phi(v).$$ We have shown that \begin{equation}\label{eq Phi preserves orthogonality among partial isometries} e,v\in \mathcal{P}I(H) \hbox{ with } e\perp v \Rightarrow \Phi(e) \perp \Phi (v).
\end{equation} Finally if $e,v \in \mathcal{P}I(H)$ with $e\leq v$, we deduce from the real linearity of $\Phi,$ \eqref{eq Phi preserves partial isometries}, and \eqref{eq Phi preserves orthogonality among partial isometries} that $\Phi (v)-\Phi(e) = \Phi(v-e)$ is a partial isometry which is orthogonal to $\Phi(e)$, therefore, $\Phi(e)\leq \Phi(v)$. This proves that $\Phi$ preserves the partial ordering in $\mathcal{P}I(H)$. The previously commented result of Moln{\'a}r  \cite[Theorem 1]{Mol2002} proves that $\Phi$ can be written in one of the forms in $(a)$ to $(d)$. If $\Phi$ is a complex linear or a conjugate linear $^*$-anti-automorphism, the mapping $\Phi^* (a) = \Phi (a^*)$ ($a\in \mathcal{B} (H)$) is a conjugate linear or complex linear $^*$-automorphism, and hence $$\Phi( \Delta_{\lambda} (a)^*) = \Phi^* \Delta_{\lambda} (a) =  \Delta_{\lambda} (\Phi^*(a)) = \Delta_{\lambda} (\Phi(a^*)) = \Delta_{\lambda} (\Phi(a)^*),$$ for all $a\in \mathcal{B} (H)$ (compare \eqref{eq linear or conjugate linear *isomorphisms preserve lambda Aluthge transforms}). Applying Proposition \ref{p basic properties}$(k)$ we derive at $$\Delta_{\lambda}(\Phi(a)^*)= \Phi(\Delta_{\lambda}(a^*)),$$ for all $a\in \mathcal{B} (H)$. Combining the last two identities with the bijectivity of $\Phi$, we deduce that $\Delta_{\lambda}(a^*) = \Delta_{\lambda} (a)^*$, for all $a\in \mathcal{B} (H)$, which is impossible, because for $a=\xi \otimes \eta,$ where $\xi$ and $\eta$ are norm-one linearly independent and non-orthogonal vectors in $H,$ \cite[Proposition 2.1]{Chabbabi2017product} implies that
$$\Delta_{\lambda}(a)^*=\langle\eta| \xi \rangle (\eta \otimes \eta)\neq \Delta_{\lambda}(a^*)=\langle \eta | \xi \rangle  (\xi\otimes \xi).$$ We have therefore proved that $\Phi$ must be a complex linear or a conjugate linear $^*$-automorphism.  \medskip\medskip

\textbf{Acknowledgements} First author supported by the Higher Education and Scientific Research Ministry in Tunisia, UR11ES52~: Analyse, G{\'e}om{\'e}trie et Applications. Second author partially supported by the Spanish Ministry of Economy and Competitiveness (MINECO) and European Regional Development Fund project no. MTM2014-58984-P and Junta de Andaluc\'{\i}a grant FQM375.
\end{document}